\documentclass[12pt, reqno]{amsart}
\usepackage{amsmath, amsthm, amscd, amsfonts, amssymb, graphicx, color}
\usepackage[bookmarksnumbered, colorlinks, plainpages]{hyperref}
\usepackage[pagewise]{lineno}
\newtheorem{theorem}{Theorem}[section]

\newtheorem{proposition}[theorem]{Proposition}
\newtheorem{corollary}[theorem]{Corollary}
\theoremstyle{definition}
\newtheorem{definition}[theorem]{Definition}

\theoremstyle{remark}
\newtheorem{remark}[theorem]{Remark}
\numberwithin{equation}{section}

\begin{document}
\setcounter{page}{1}

\title[Algebraic orthogonality and commuting projections]{Algebraic orthogonality and commuting projections   
in operator algebras}

\author[A. K. Karn]{Anil Kumar Karn}

\address{School of Mathematical Science, National Institute of Science Education and 
Research, HBNI, Bhubaneswar, At $\&$ Post - Jatni, PIN - 752050, India.}
\email{\textcolor[rgb]{0.00,0.00,0.84}{anilkarn@niser.ac.in; anil.karn@gmail.com}}



\begin{abstract} 
We provide an order-theoretic characterization of algebraic 
orthogonality among positive elements of a general C$^{\ast}$-algebra 
by proving a statement conjectured in \cite{AK2}. Generalizing this
idea, we describe absolutely ordered $p$-normed spaces, for 
$1 \le p \le \infty$ which present a model for ``non-commutative 
vector lattices''. Thid notion includes order theoretic orthogonality. We 
generalize algebraic orthogonality by introducing the notion of 
{\it absolute compatibility} among positive elements in absolute order 
unit spaces and relate it to symmetrized product in the case of a 
C$^{\ast}$-algebra. In the latter case, whenever one of the elements is 
a projection, the elements are absolutely compatible if and only if they 
commute. We develop an order theoretic prototype of the results. For 
this purpose, we introduce the notion of {\it order projections} and 
extend the results related to projections in a unital C$^{\ast}$-algebra 
to order projections in an absolute order unit space. As an application, 
we describe spectral decomposition theory for elements of an absolute 
order unit space.
\end{abstract} 

\maketitle

\section{Introduction}
The order structure of a C$^{\ast}$-algebra has been a point of 
attraction since the inception of the theory. Kakutani's characterization 
of $C(K)$ spaces ($K$ a compact, Hausdorff space) as $AM$- spaces 
\cite{SK} highlighted that the self-adjoint part of a commutative 
C$^{\ast}$-algebra is a Banach lattice (with some additional norm 
conditions). However, in a non-commutative C$^{\ast}$-algebra, 
join and meet of two general self-adjoint elements may not even exist. 
Thus it was natural to turn attention towards the non-commutative 
case. In 1951, Kadison proved that the self-adjoint part of a unital 
C$^{\ast}$-algebra  is an order unit space \cite{RVK}. (However, this 
was not a characterization as the converse is not true.) Its non-unital 
version was studied by Ng \cite{KFN}. (Also see \cite{LA}.) The 
non-self-adjoint version of Kadison's work was introduced by Choi and 
Effros as matrix order unit spaces \cite{CE}  whose non-unital version 
was presented by Karn and Vasudevan \cite{KV}.

The author carried forward the work further in this direction with an 
intuition that it may be possible to prove a non-commutative version 
of Kakutani's theorem. He characterized the (matrix) ordered normed 
spaces that can be order embedded in C$^{\ast}$-algebras \cite{AK10} 
and introduced the notion of order smooth $\infty$-normed spaces 
(order smooth $p$-normed spaces in general, for $1 \le p \le \infty$) 
\cite{AK}. On such spaces, he studied a notion pf 
$\infty$-orthogonality ($p$-orthogonality in general, for $1\le p \le 
\infty$) \cite{AK1}. In a subsequent paper, he characterized algebraic 
orthogonality in some classes of C$^{\ast}$-algebras (that include 
commutative C$^{\ast}$-algebras as well as von Neumann algebras) in 
terms of absolute $\infty$-orthogonality (defined for order smooth 
$\infty$-normed spaces) \cite{AK2}. In this paper, we extend it to an 
arbitrary C$^{\ast}$-algebra, thus proving Conjecture 4.4 of \cite{AK2}.

Following the above said characterization, the author introduced the 
notion of an absolute order smooth $p$-normed space ($1 \le p \le 
\infty$). Thus the examples of an absolute order smooth 
$\infty$-normed space include the self-adjoint part of an arbitrary 
C$^{\ast}$-algebra. It is important to note that an absolute order 
smooth $p$-normed space exhibit a ``vector lattice like'' structure. 
More precisely, this structure can be characterized as a vector lattice 
under an extra condition \cite{AK2}. In this paper, we shall present a 
simplified version of this theory to propose a model of a 
``non-commutative'' vector-lattice theory. 

Algebraic orthogonal (or equivalently, absolutely orthogonal) pair of 
positive elements in a C$^{\ast}$-algebra are by default commutative.  
In this paper, we observe that absolutely orthogonal pair of (positive) 
elements inherit another order theoretic relation which we term as 
absolute compatibility. We show that for a pair of absolutely 
compatible positive elements in a C$^{\ast}$-algebra, their 
symmetrized product may be described order theoretically. More 
precisely, we show that in a C$^{\ast}$-algebra $A$,  $a$ and $b$ are 
mutually absolutely compatible positive elements of $A$ if and 
only if $\alpha (a \dot\wedge b) = a \circ b$ where $\alpha = \max \{ 
\Vert a \Vert, \Vert b \Vert \}$. (Notions are defined later.) In particular, 
if one of the elements is a projection, then these elements are 
absolutely compatible if and only if they commute. These observations 
indicate that absolute compatibility may be explored as a possible tool 
to understand commutativity in operator algebras.

In this paper, we develop an order theoretic prototype of these results. 
For this purpose, we introduce the notion of {\it order projections} 
generalizing projections and extend some of the results related to 
projections in unital C$^{\ast}$-algebras to order projections in 
absolute order unit spaces. Order projections bear similarity with the 
notion `projective units' (defined in  order unit spaces) studied in 
\cite{AS} and also with the notion `projections' (again defined in order 
unit spaces) studied in \cite{PB, AM}. At the end of the paper, as an 
application, we describe a spectral decomposition theory in an 
absolute order unit space.

Now we propose the scheme of the paper. In Section 3, we describe 
absolutely ordered $p$-normed spaces, for $1 \le p \le \infty$ which 
presents a model for ``non-commutative'' vector lattices and includes 
order theoretic orthogonality. In section 4, we introduce {\it absolute 
compatibility} between positive elements in an absolute order 
unit space and relate this notion to symmetrized product in a unital 
C$^{\ast}$-algebra. In Section 5, we introduce {\it order projections} as 
a generalization of projections in operator algebras. We study absolute 
compatibility of an order projection first, with another order projection 
in Section 5, and then with general positive elements in Section 6. In 
Section 7, we discuss a spectral decomposition theory in an absolute 
order unit space.

\section{Orthogonality in C$^{\ast}$-algebras}

In \cite{AK2}, we proved that the algebraic orthogonality among 
positive elements is equivalent to absolutely $\infty$-orthogonality 
in a von Neumann algebra as well as in a commutative 
C$^{\ast}$-algebra. We begin the paper with proving the result for a general C$^{\ast}$-algebra conjectured in \cite{AK2}.
\begin{theorem}\cite[Conjecture 4.4]{AK2}\label{001}
	Let $A$ be a C$^{\ast}$-algebra. Then for $a, b \in A^+ \setminus\{ 0 \}$, we have $a b = 0$ if and only if 
	$\Vert \Vert c \Vert^{-1} c + \Vert d \Vert^{-1} d \Vert = 1$ whenever $0 < c \le a$ and $0 < d \le b$ in $A^+$.
\end{theorem} 

Let us recall the following result.

\begin{proposition}\cite{AK2}\label{8}
Let $A$ be a $C^{\ast}$-algebra and let $a, b \in A^+ \setminus \{ 0 \}$. 
Consider the following statements:
 \begin{enumerate}
 \item[(1)] $a b = 0$, 
 \item[(2)] $0 < c \le a$ and $0 < d \le b$ imply $\big\Vert \Vert c 
 \Vert^{-1} c + \Vert d \Vert^{-1} d \big\Vert = 1$,
 \item[(3)] $0 \le c \le a$ and $c \le b$ imply $c = 0$. 
\end{enumerate}
Then $(1)$ implies $(2)$ and $(2)$ implies $(3)$. Further, if $a b = b 
a$, then $(3)$ implies $(1)$. 
\end{proposition}

\begin{proof} 
	
	[of Theorem \ref{001}] 
	( A. M. Peralta)
	It suffices to show that $a \perp_{\infty}^a b$ implies $a b = 0$. Further, without any loss of generality, we may assume that $\Vert a \Vert = 1 = 
	\Vert b \Vert$. 
	Let $C^{\ast}(a)$ be the C$^{\ast}$-subalgebra of $A$ generated by $a$. Then $C^{\ast}(a) \cong C(\sigma (a))$ where $\sigma (a)$ is the 
	spectrum of $a$. Since $\Vert a \Vert = 1$ we have $\sigma (a) \subset [0, 1]$. 
	For each $n \in \mathbb{N}$, we define $a_n: \sigma (a) \to \mathbb{C}$ as follows:  
	For $t \in \sigma (a)$, we set 
	\begin{eqnarray*}
		a_n (t) &=& t \quad \textrm{for}~ t \le \frac{1}{n} \\
		&=& \frac{1}{n} \quad \textrm{for}~ t \ge \frac{1}{n}.
	\end{eqnarray*}
	
	Then $a_n \in C(\sigma (a))$ with $a_n \le a$ for each $n$. By functional calculus, $a_n \in A^+$. Thus by assumption, $a_n \perp_{\infty} d$ 
	for any $d \in A^+$ with $d \le b$. 
	Also, $\Vert a_n \Vert = \frac{1}{n}$ so that $c_n := n a_n$ has norm one for each $n$ and that 
	$c_n \perp_{\infty} d$ for any $d \in A^+$ with $d \le b$. 
	Further we note that $c_n \to [a]$ in $A^{\ast\ast}$ in the weak*-topology where $[a]$ is the 
	range projection of a in $A^{\ast\ast}$.  
	As the norm in $A^{\ast\ast}$ is weak*-lower semi-continuous we have 
	$$\Vert [a] + \Vert d \Vert^{-1} d \Vert \le \lim_{w*} \Vert c_n + \Vert d \Vert^{-1} d \Vert = 1$$ 
	for $0 \le d \le b$. Now as $0 \le [a] \le [a] + \Vert d \Vert^{-1} d$ we may conclude that $\Vert [a] + \Vert d \Vert^{-1} d \Vert = 1$ we have 
	$[a] \perp_{\infty} d$ whenever $0 \le d \le b$. 
	Now by a dual argument, we further get that $[a] \perp_{\infty} [b]$. 
	As $[a]$ and $[b]$ are projections, we have 
	$[a] [b] = 0$ and consequently, $a b = 0$.
\end{proof}

\section{Orthogonality in ordered vector spaces}

In this section, we recall few immediate definitions and facts discussed in \cite{AK, AK1, AK2}. we shall present these concepts 
with a new orientation. This may be seen as a fresh start of the theory of {\it absolutely ordered spaces}. The first result is a 
simpler (and weaker) form of \cite[Theorem 4.11]{AK2}. We include a proof as the order structure is proved under weaker assumptions and with a different set of arguments.
\begin{theorem} \label{1}
Let $V$ be a real vector space. The following sets of conditions on $V$ are equivalent:
\begin{enumerate}
\item There exists a cone $V^+$ in $V$ and a mapping $\vert \cdot \vert : V \to V^+$ that satisfies the following conditions:
\begin{enumerate}
\item $\vert v \vert = v$ if $v \in V^+$.
\item $\vert v \vert \pm v \in V^+$.
\item $\vert k v \vert = \vert k \vert \vert v \vert$ for all $v \in V$ and $k \in \mathbb{R}$.
\end{enumerate}
\item There exists a mapping $\dot\vee : V \times V \to V$ that satisfies the following conditions:
\begin{enumerate}
\item $v \dot\vee v = v$.
\item $v \dot\vee w = w \dot\vee v$ for all $v, w \in V$.
\item $(u \dot\vee v) + w = (u + w) \dot\vee (v + w)$ for all $u, v, w \in V$.
\item $k (v \dot\vee w) = (k v) \dot\vee (k w)$ for all $v, w \in V$ and $ k \ge 0$.
\item If $v \dot\vee w = v$, then $(u \dot\vee v) \dot\vee w = u \dot\vee (v \dot\vee w)$.
\end{enumerate}
\item There exists a mapping $\dot\wedge : V \times V \to V$ that satisfies the following conditions:
\begin{enumerate}
\item $v \dot\wedge v = v$.
\item $v \dot\wedge w = w \dot\wedge v$ for all $v, w \in V$.
\item $(u \dot\wedge v) + w = (u + w) \dot\wedge (v + w)$ for all $u, v, w \in V$.
\item $k (v \dot\wedge w) = (k v) \dot\wedge (k w)$ for all $v, w \in V$ and $ k \ge 0$.
\item If $v \dot\wedge w = v$, then $(u \dot\wedge v) \dot\wedge w = u \dot\wedge (v \dot\wedge w)$.
\end{enumerate}
\item There exists a cone $V^+$ in $V$ and a binary operation $\perp$ in $V^+$ that satisfies the following conditions:
\begin{enumerate}
\item $u \perp 0$ for all $u \in V^+$.
\item If $u \perp v$, then $v \perp u$.
\item If $u \perp v$, then $k u \perp k v$ for all $k \in \mathbb{R}$ with $k > 0$.
\item For each $v \in V$, there exist a unique pair $v_1, v_2 \in V^+$ such that $v = v_1 - v_2$ with 
$v_1 \perp v_2$.
\end{enumerate}

\end{enumerate}
\end{theorem}
\begin{proof}
First assume that the set of conditions (1) holds. For $v, w \in V$, we define
$$v \dot\vee w = \frac{1}{2} (v + w + \vert v - w \vert ).$$
Then $\dot\vee : V \times V \to V$ and (2) (a), (b), (c) and (d) follow in a routine way. Further note that $v \dot\vee w = v$ if and only if 
$w \le v$. Also, $v \le u \dot\vee v$. Now, we show that (2)(e) holds. Let $u, v, w \in V$ with $v \dot\vee w = v$. Then $w \le v \le u 
\dot\vee v$ so that $(u \dot\vee v) \dot\vee w = u \dot\vee v = u \dot\vee (v \dot\vee w)$. Thus (1) implies (2).

Next, assume that the set of conditions (2) holds. Set $v \dot\wedge w := v + w - (v \dot\vee w)$. Then $\dot\wedge: V \times V \to V$ 
is a binary mapping such that
$$v \dot\wedge w = v + w - (v \dot\vee w) = - \{ (- v) \dot\vee (- w) \} = \frac{1}{2} (v + w - \vert v - w \vert )$$
for all $v, w \in V$. Now conditions (3)(a) -- (3)(e) hold by dual arguments.

Finally, assume the conditions in the set (3). Put
$$V^+ = \{ v: v\dot\wedge 0 = 0 \}.$$
Let $v \in V^+$ and $k \ge 0$. Then $v \dot\wedge 0 = 0$ so that using (2)(d), we get $0 = k (v \dot\wedge 0) = (k v) \dot\wedge 0$. Thus $k v \in V^+$. Next, note that by using (2)(c) and (2)(d) we can show that $u in V^+$, that is, $u \dot\wedge 0 = 0$ if and only $(- u) \dot\wedge 0 = - u$ which is equivalent to $u \dot\wedge (- u) = - u$. Now, let $v, w \in V^+$. Then $- v = v \dot\wedge (- v)$ and 
$- w = w \dot\wedge (- w)$so that $- v - w = (v - w) \dot\wedge (- v - w)$ by (2)(c). Similarly, we get $v - w = (v + w) \dot\wedge (v - w)$. 
Thus using (2)(e), we have

\begin{align*}
(v + w) \dot\wedge (- v - w) &= (v + w) \dot\wedge \{ (v - w) \dot\wedge (- v - w) \} \\
&= \{ (v + w) \dot\wedge (v - w) \} \dot\wedge (- v - w) \\
&= (v - w) \dot\wedge (-v - w) \\
&= - v - w.
\end{align*}
Therefore, $V^+$ is a cone in $V$. Now we define 
$$\vert v \vert = - (v \dot\wedge (- v))$$
for all $v \in V$. Note that $v \in V^+$ if and only if $0 = 2 (v \dot\wedge 0) = (2 v) \dot\wedge 0 = - \vert v \vert  + v$ so that (1)(a) holds. 

To prove (1)(b), first we show that $-(v \dot\wedge 0) \in V^+$ for all $v \in V$. For this, let $v \in V$ and 
set $w = v \dot\wedge 0$. Then
$$w \dot\wedge 0 = ( v \dot\wedge 0 ) \dot\wedge 0 = v \dot\wedge 0 = w$$
by (2)(e) so that
$$w \dot\wedge (-w) = (2 w) \dot\wedge 0 - w = 2 ( w \dot\wedge 0) - w = w$$
by (2)(c) and (2)(d). Thus
$$- w = \vert w \vert = \vert (- w) \vert$$
and whence $-( v \dot\wedge 0) = - w \in V^+$. Now (1)(b) follows from a straight forward observation 
$$\vert v \vert \pm v = - 2 ( ( \mp v) \dot\wedge 0).$$
Now, the proof of (1)(c) directly follows from (2)(d).

Finally, we show the equivalence of (1) and (4). First let the set of conditions (1) hold. For $u, v \in V^+$, we define define $u \perp v$ 
if $\vert u - v \vert = u + v$. Then conditions (4)(a) -- (4)(d) directly follow from the definition of $\perp$.
Conversely, assume that the set of conditions (4) hold. For each $v$, define $\vert v \vert := v_1 + v_2$, using the uniqueness of (4)(d). 
Then $\vert\cdot\vert$ maps $V$ into $V^+$. Let $v \in V^+$. Since $v \perp 0$ by (4)(a), we have $\vert v \vert = v$. Also, by the definition of 
$\vert\cdot\vert$, we further see that $\vert v \vert \pm v \in V^+$.  Now, let $v \in V$. Then by (4)(d), there exists a unique pair 
$v_1, v_2 \in V^+$ with $v_1 \perp v_2$ such that $v = v_1 - v_2$. Then $k v_1 \perp k v_2$ for any $k \in \mathbb{R}$ with 
$k > 0$ using condition (4)(c). Now by the definition
$$\vert k v \vert = \vert (k v_1 - k v_2) \vert = k v_1 + k v_2 = k \vert v \vert.$$
Also, if $k \in \mathbb{R}$ with $k < 0$, then 
$$\vert k v \vert = \vert ((-k) v_2 - (-k) v_1) \vert = (-k) v_2 + (-k) v_1 = (-k) \vert v \vert.$$
Thus $\vert k v \vert = \vert k \vert \vert v \vert$ for all $k \in \mathbb{R}$.
\end{proof} 

Next, we recall Theorem 4.12 of \cite{AK2}.
\begin{theorem} \label{3}
Let $V$ satisfy the (equivalent sets of) conditions of Theorem \ref{1}. Then the following statements are equivalent:
\begin{enumerate}
\item[(i)] $v \dot\vee w = \sup \{ v, w \}$ for all $v, w \in V$.
\item[(ii)] $\dot\vee$ is associative in $V$.
\item[(iii)] $u \pm v \in V^+$ implies $\vert v \vert \le u$ for all $u, v \in V$.
\item[(iv)] $\vert v + w \vert \le \vert v \vert + \vert w \vert$ for all $v, w \in V$.
\end{enumerate}
Thus $V$ is a vector lattice if and only if one of the equivalent conditions of this result (in addition to the equivalent set of 
conditions of Theorem \ref{1}) holds in $V$. 
\end{theorem}

\begin{remark}\label{4} 
In addition to the (equivalent sets of) conditions of Theorem \ref{1}, there are some other properties which hold both in a vector lattice 
as well as in a C$^{\ast}$-algebra.

Let $V$ be a vector lattice.
\begin{enumerate}
\item If $u, v, w \in V$ with $u \wedge v = 0$ and $u \wedge w = 0$, then $u \wedge (v + w) = 0$. To see 
this, first note that $x \wedge y \le x$ for any $x, y \in V$ so that $u, v, w \in V^+$. Now, as $u \wedge v = 0$, we get $(u + w) \wedge 
(v + w) = w$. Thus 
$$0 = u \wedge w = u \wedge (u + w) \wedge (v + w) = u \wedge (v + w)$$ 
for $u \le u + w$.
\item Let $u, v \in V$ with $u \wedge v = 0$. If $0 \le w \le v$, then $u \wedge w = 0$. In fact, $w \wedge v 
= w$ so that 
$$u \wedge w = u \wedge v \wedge w = 0 \wedge w = 0.$$
In particular, if $u, v, w \in V^+$ with $u \wedge (v + w) = 0$, then $u \wedge v$ and $u \wedge w = 0$.
\item If $u, v, w \in V$ with $u \wedge v = 0$ and $u \wedge w = 0$, then $u \wedge \vert v - w \vert = 0$. 
This follows from (1) and (2) as $\vert v - w \vert \le v + w$.

Next, let $A$ be a C$^{\ast}$-algebra.
\item For $a \in A$ we define $\vert a \vert := (a^{\ast} a)^{\frac{1}{2}}$. Then $(A_{sa}, A^+, \vert\cdot\vert )$ satisfies the set of conditions 
(1) of Theorem \ref{1}. Also,  (1), (2) and (3) hold in $A_{sa}$ as well when we replace $\wedge$ by $\dot\wedge$. In fact, we have 
$a \dot\wedge b = 0$ if and only if $a, b \in A^+$ and $a b = 0$. To see this, first we note that $a \dot\wedge b = 0$ if and only if 
$\vert a - b \vert = a + b$. We show that $\vert a - b \vert = a + b$ if and only if $a, b \in A^+$ with $a b = 0$. First, let  $\vert a - b \vert = a + b$. 
Then 
$$a^2 - a b - b a + b^2 = \vert a - b \vert^2 = (a + b)^2 = a^2 + a b + b a + b^2$$
so that $a b + b a = 0$. Also, in this case, 
$$(a + b) \pm (a - b) = \vert a - b \vert \pm (a - b) \in A^+$$
so that $a, b \in A^+$. Thus $-a^2 b = a b a \in A^+$ and consequently, $a^2 b = b a^2$. Now, by the functional calculus in $A$, we 
can conclude that $a b = b a$. Consequently, $a b = 0$. Conversely, if $a, b \in A^+$ 
with $a b = 0$, then $a b = b a$ and using the (commutative) C$^{\ast}$-algebra generated by $a$ and $b$, we can show that 
$\vert a - b \vert =  a + b$ so that $a \dot\wedge b = 0$.
\item We have not been able to prove whether (1), (2) and (3) (with $\wedge$ replaced by $\dot\wedge$) will follow from the equivalent 
conditions of Theorem \ref{1} or not. Also, we find these properties useful for the develop the theory. Thus we shall include them in the 
definition of an {\it absolutely ordered space}. 
\item Let $V$ satisfy the equivalent conditions of Theorem \ref{1} and let $u, v \in V$. Then $u \dot\wedge v = 0$ if and only if 
$u, v \in V^+$ and $\vert u - v \vert = u + v$ (or equivalently, $u \perp v$). Further, $\vert 2 u - v \vert = v$ if and only if 
$0 \le u \le v$ and $u \perp (v - u)$.
\end{enumerate}
\end{remark}
\begin{definition} \label{2}
Let $(V, V^+)$ be a real ordered vector space and let $\vert\cdot\vert: V \to V^+$ be a mapping satisfying the following conditions:  
\begin{enumerate}
\item $\vert v \vert = v$ if $v \in V^+$; 
\item $\vert v \vert \pm v \in V^+$; 
\item $\vert k v \vert = \vert k \vert \vert v \vert$ for all $v \in V$ and $k \in \mathbb{R}$; 
\item If  $\vert u - v \vert = u + v$ and $0 \le w \le v$, then $\vert u - w \vert = u + w$; and
\item If $\vert u - v \vert = u + v$ and $\vert u - w \vert = u + w$, then $\vert u - ( v \pm w ) \vert = u + \vert v \pm w \vert$.
\end{enumerate} 
Then $(V, V^+, \vert\cdot\vert )$ will be called an {\emph absolutely ordered space}. 
\end{definition}

\subsection{Norms on absolutely ordered vector spaces}

\begin{definition}
Let $\left( V, V^+ \right)$ be a real ordered vector space such that $V^+$ is proper and generating and  
let $\Vert \cdot \Vert$ be a norm on $V$ such that $V^+$ is $\Vert \cdot \Vert$-closed. Let $1 \le p \le \infty$. 
For $u, v \in V^+$, we say that $u$ is $p$-orthogonal to $v$, (we write, $u \perp_p v$), if 
\begin{align*}
\Vert \alpha u + \beta v \Vert &= (\Vert \alpha u \Vert^p + \Vert \beta v \Vert^p )^{\frac{1}{p}} \quad 
(1 \le p < \infty ) \\
&= \max (\Vert \alpha u \Vert , \Vert \beta v \Vert ) \quad (p = \infty ).
\end{align*}
Further, we say that $u$ is absolutely $p$-orthogonal to $v$, (we write, $u \perp_p^a v$), if $u_1 \perp_p v_1$ 
whenever $0 \le u_1 \le u$ 
and $0 \le v_1 \le v$.
\end{definition}
The following observation describes the importance of $\perp_p$.
\begin{proposition}\label{6}
Let $(V, V^+, \vert\cdot\vert)$ be an absolutely ordered vector space and assume that $\Vert\cdot\Vert$ 
is a norm on $V$ such that $V^+$ is $\Vert\cdot\Vert$-closed. Then for $1 \le p \le \infty$, the following 
conditions are equivalent:
\begin{itemize}
\item[(A)] For each $v \in V$, we have 
\begin{align*}
\Vert \vert v \vert \Vert = \Vert v \Vert &= \left( \Vert v^+ \Vert^p + \Vert v^- \Vert^p \right)^{\frac{1}{p}} \quad 
(1 \le p < \infty )\\
&= \max ( \Vert v^+ \Vert, \Vert v^- \Vert ) \quad (p = \infty );
\end{align*}
\item[(B)] For $u, v \in V^+$, we have $u \perp_p^a v$ whenever $u \perp v$;
\item[(C)] For $u, v \in V^+$, we have $u \perp_p v$ whenever $u \perp v$.
\end{itemize}
If $\Vert\cdot\Vert$ is an order unit norm determined by the order unit $e$, then the above conditions (with $p = \infty$) 
are also equivalent to:
\begin{itemize}
\item[(D)] For each $v \in V$ with $\pm v \le e$, we have $\vert v \vert \le e$.
\end{itemize}
\end{proposition}
\begin{proof}
First, assume that (A) holds. Let $u, v \in V^+$ with $u \perp v$ and suppose that $0 \le u_1 \le u$ and 
$0 \le v_1 \le v$. Then for $k, l \in \mathbb{R}$ with $k, l > 0$, we have $k u_1 \perp l v_1$. If we set 
$w = k u_1 - l v_1$, then $\vert w \vert = k u_1 + l v_1, w^+ = k u_1$ and $w^- = l v_1$. Thus by (A), 
we have
\begin{align*}
\Vert k u_1 \pm l v_1 \Vert &= (\Vert k u_1 \Vert^p + \Vert l v_1 \Vert^p )^{\frac{1}{p}} \quad (1 \le p < \infty ) \\
&= \max (\Vert k u_1 \Vert , \Vert l v_1 \Vert ) \quad (p = \infty )
\end{align*}
so that $u \perp_p^a v$. Thus (A) implies (B). Now that (B) implies (C) is trivial. 

Next, assume that (C) holds. 
Let $v \in V$. Then $v^+ \perp v^-$. Thus by assumption, $v^+ \perp_p v^-$ so that 
\begin{align*}
\Vert v^+ \pm v^- \Vert &= (\Vert v^+ \Vert^p + \Vert v^- \Vert^p )^{\frac{1}{p}} \quad (1 \le p < \infty ) \\
&= \max (\Vert v^+ \Vert , \Vert v^- \Vert ) \quad (p = \infty ).
\end{align*}
Since $v = v^+ - v^-$ and $\vert v \vert = v^+ + v^-$, we see that (A) holds. 

Now, assume that $\Vert\cdot\Vert$ is an order unit norm 
determined by the order unit $e$. First let (D) hold. Let $u, v \in V^+$ with $u \perp v$. We show that 
$u \perp_{\infty} v$. Without any loss of generality, we may assume that $\Vert u \Vert = 1 = \Vert v \Vert$. Set $w = u - v$. As 
$- v \le w \le u$, we have $\Vert w \Vert \le \max \{ \Vert u \Vert, \Vert v \Vert \} = 1$ so that $\pm w \le e$. Now, by assumption, 
$u + v = \vert w \vert \le e$. Thus $1 = \Vert u \Vert \le \Vert u + v \Vert \le 1$ so that $u \perp_{\infty} v$. 

Finally,  assume that (A) holds. Let $v \in V$ with $\pm v \le e$. Then 
$\Vert v \Vert \le 1$. Now, by assumption, $\Vert \vert v \vert \Vert \le 
1$ so that $\vert v \vert \le e$. This completes the proof.
\end{proof}
\begin{definition}
Let $(V, V^+, \vert\cdot\vert)$ be an absolutely ordered vector space 
and assume that $\Vert\cdot\Vert$ is a norm on $V$ such that $V^+$ 
is $\Vert\cdot\Vert$-closed. For $1 \le p \le \infty$, we say that 
$(V, V^+, \vert\cdot\vert, \Vert\cdot\Vert )$ is an {\it absolutely order 
smooth $p$-normed space}, if it satisfies the following conditions.
\begin{itemize}
 \item[$(O.p.1)$:] For $u \le v \le w$ we have 
 \begin{align*}
\Vert v \Vert &\le (\Vert u \Vert^p + \Vert w \Vert^p )^{\frac{1}{p}} \quad (1 \le p < \infty ) \\
&\le \max (\Vert u \Vert , \Vert w \Vert ) \quad (p = \infty ); \qquad \textrm{and}
\end{align*}
\item[$(O.\perp_p.1)$:] For $u, v \in V^+$ with $u \perp v$, we have $u \perp_p^a v$.
\item[$(O.\perp_p.2)$:] For $u, v \in V^+$ with $u \perp_p^a v$, we have $u \perp v$.
\end{itemize} 
If in addition, (for $p = \infty$,)  $\Vert\cdot\Vert$ is an order unit 
norm on $V$ determined by an order unit $e$, we say that $(V, e)$ is 
an {\it absolute order unit space}.
\end{definition}
\begin{remark}\label{7} \hfill
\begin{enumerate} 
\item The self-adjoint part of every C$^{\ast}$-algebra is an absolutely order smooth $\infty$-normed space.
\item The self-adjoint part of the dual of any C$^{\ast}$-algebra is an absolutely order smooth $1$-normed space.
\item In general,  $T_p(H)_{sa}$, the self-adjoint part of the space of trace class operators on a complex Hilbert space $H$ is an absolutely 
order smooth $p$-normed space.
\item Let $(V, V^+, \vert\cdot\vert, \Vert\cdot\Vert )$ is an absolutely order smooth $p$-normed space. Assume that $u, v, w \in V^+$ 
with $u \perp v$ and $u \perp w$. Then for $\alpha, \beta > 0$, we have $u \perp \alpha v$ and $u \perp \beta w$ and consequently, 
$u \perp \vert \alpha v + \beta w \vert$. Thus $u \perp_p (\alpha v + \beta w)$ for all $\alpha, \beta \in \mathbb{R}$.
\item Recall that an order smooth $p$-normed space is a real ordered vector space $\left( V, V^+ \right)$ in which $V^+$ is proper and 
generating together with a norm $\Vert \cdot \Vert$ for which $V^+$ is closed such that the space satisfies the conditions 
$(O.p.1)$ and \\
$(O.p.2)$: For each $v \in V$ and $\epsilon > 0$, there exist $v_1, v_2 \in V^+$ such that $v = v_1 - v_2$ and 
 \begin{align*}
\Vert v \Vert + \epsilon &> (\Vert v_1 \Vert^p + \Vert v_2 \Vert^p )^{\frac{1}{p}} \quad (1 \le p < \infty ) \\
&> \max (\Vert v_1 \Vert , \Vert v_2 \Vert ) \quad (p = \infty ).
\end{align*} 
Thus an absolutely order smmoth $p$-normed space is an order smooth $p$-normed space satisfying \\
$(OS.p.2)$: For each $v \in V$, there exist $v_1, v_2 \in V^+$ such that $v = v_1 - v_2$ and 
 \begin{align*}
\Vert v \Vert &= (\Vert v_1 \Vert^p + \Vert v_2 \Vert^p )^{\frac{1}{p}} \quad (1 \le p < \infty ) \\
&= \max (\Vert v_1 \Vert , \Vert v_2 \Vert ) \quad (p = \infty ).
\end{align*} 
\item Let $(V, V^+, e)$ be an order unit space. Then, with the order unit norm $\Vert\cdot\Vert_e$ on $V$, $(V, V^+, \Vert\cdot\Vert_e )$ 
is an order smooth $\infty$-normed space satisfying $(OS. \infty . 2)$.
\end{enumerate}
\end{remark}

\section{Absolute orthogonality and symmetrized product}
By the definition, algebraic orthogonal pair of positive elements in a 
C$^{\ast}$-algebra commute.  In this section, we introduce an order 
theoretic notion, defined forpositive elements, which nearly imitates 
commutativity when considered in a C$^{\ast}$-algebra. 
\begin{proposition}\label{17}
	Let $(V, e)$ be an absolute order unit space and assume that $u, v \in [0, e]$. Then $u \perp v$ if and only if $u + v \le e$ and 
	$\vert u - v \vert + \vert u + v - e \vert =  e$.
\end{proposition}
\begin{proof}
	Let $u \perp v$. Then $\vert u - v \vert = u + v$ and $\Vert u + v \Vert = \max \{ \Vert u \Vert, \vert v \Vert \} \le 1$. Thus $u + v \in [0, e]$ so that 
	$$\vert u + v - e \vert = e - (u + v) = e - \vert u - v \vert.$$ 
	Now, it follows that $\vert u - v \vert + \vert u + v - e \vert =e$. Conversely, assume that $u + v \le e$ and that $\vert u - v \vert + \vert u + v - e \vert =e$. Then 
	$$e - \vert u - v \vert = \vert u + v - e \vert = e - (u + v)$$
	so that $\vert u - v \vert = u + v$. In other words, $u \perp v$.
\end{proof}

\begin{proposition}\label{14}
Let $(V, e)$ be an absolute order unit space. If $\vert u - v \vert + \vert u + v - e \vert = e$, then $u, v \in [0, e]$. 
\end{proposition}
\begin{proof}
Assume that $\alpha = 1$. Then $\vert u - v \vert + \vert u + v - e \vert = e$. Thus we have  
\begin{eqnarray*}
0 &\le& \vert u - v \vert \pm (u - v) + \vert u + v - e \vert \pm (u + v - e) \\ 
&=& e \pm (u - v) \pm (u + v - e).
\end{eqnarray*} 
Thus 
$$e + (u - v) + (u + v - e) \ge 0,$$
$$e - (u - v) + (u + v - e) \ge 0,$$
$$e + (u - v) - (u + v - e) \ge 0$$
and
$$e - (u - v) - (u + v - e) \ge 0.$$
Now, it follows that $u, v \in [0, e]$.
\end{proof}
\begin{definition}
	Let $(V, e)$ be an absolute order unit space. Then $u, v \in [0, e]$ are said to be {\it absolutely comparable}, if 
	$\vert u - v \vert + \vert u + v - e \vert = e$. 
\end{definition}
\begin{proposition}\label{15}
Let $(V, e)$ be an absolute order unit space. Then for $u, v \in [0, e]$ the following statements are equivalent:
\begin{enumerate}
\item $u$ is absolutely compatible with $v$;
\item $u \dot\wedge v + u \dot\wedge (e - v) = u$;
\item $u \dot\wedge v + (e - u) \dot\wedge v = v$; 
\item $(e - u) \dot\wedge v + (e - u) \dot\wedge (e - v) = e - u$;
\item $u \dot\wedge (e - v) + (e - u) \dot\wedge (e - v) = e - v$;
\item $u \dot\wedge v + u \dot\wedge (e - v) + (e - u) \dot\wedge v + (e - u) \dot\wedge (e - v) = e$.
\end{enumerate}
\end{proposition}
\begin{proof}
Using symmetry in the condition for absolute compatibility, we may conclude that $u$ is absolutely compatible with $v$ if and only if 
$\{ u, e - u \}$ is absolutely compatible with $\{ v, e - v \}$. Next, as 
\begin{eqnarray*}
u \dot\wedge v + u \dot\wedge (e - v) &=& \frac{1}{2} \{ u + v - \vert u - v \vert + u + e - v - \vert u - e + v \vert \} \\
&=& u + \frac{1}{2} \{ e - ( \vert u - v \vert + \vert u + v - e \vert ) \},
\end{eqnarray*} 
we conclude that $u$ and $v$ are absolutely compatible with respect to $p$ if and only if $u \dot\wedge v + u \dot\wedge (e - v) = u$. 
Now combining the two observations, the proofs follow in a routine way.
\end{proof}
\begin{proposition}\label{16}
Let $(V, e)$ be an absolute order unit space. Then for $u, v \in [0, e]$ the following statements are equivalent:
\begin{enumerate}
\item $u$ is absolutely compatible with $v$;
\item $u \dot\wedge v, (e - u) \dot\wedge (e - v) \in V^+$ with $u \dot\wedge v \perp (e - u) \dot\wedge (e - v)$;
\item $u \dot\wedge (e - v), (e - u) \dot\wedge v \in V^+$ with $u \dot\wedge (e - v) \perp (e - u) \dot\wedge v$.
\end{enumerate}
\end{proposition}
\begin{proof}
First, let us assume that $u$ is absolutely compatible with $v$ so that $\vert u - v \vert + \vert u + v - e \vert = e$. Then  
\begin{eqnarray*}
u \dot\wedge v &=& \frac{1}{2} \{ u + v - \vert u - v \vert \} \\ 
&=& \frac{1}{2} \{ u + v - e + \vert u + v - e \vert \} \\
&=& (u + v - e)^+.
\end{eqnarray*}
In a similar manner, we can also show that 
$$u \dot\wedge (e - v) = (u - v)^+;$$
$$(e - u) \dot\wedge v = (u - v)^-;$$
and 
$$(e - u) \dot\wedge (e - v) = (u + v - e)^-.$$
Thus $u \dot\wedge v, u \dot\wedge (e - v), (e - u) \dot\wedge v, (e - u) \dot\wedge (e - v) \in V^+$ with 
$$u \dot\wedge v \perp (e - u) \dot\wedge (e - v)$$ 
and 
$$u \dot\wedge (e - v) \perp (e - u) \dot\wedge v.$$ 
Thus (1) implies (2) and (3).

Conversely, let $u \dot\wedge v, (e - u) \dot\wedge (e - v) \in V^+$ with $u \dot\wedge v \perp (e - u) \dot\wedge (e - v)$. Then by the definition, we have 
$$u \dot\wedge v + (e - u) \dot\wedge (e - v) = e - \vert u - v \vert$$ 
and 
$$u \dot\wedge v - (e - u) \dot\wedge (e - v) = u + v - e.$$
Now, $u \dot\wedge v \perp (e - u) \dot\wedge (e - v)$ implies that $\vert u + v - e \vert = e - \vert u - v \vert$, that is, $u$ is absolutely compatible with $v$. 

Dually, $u$ is absolutely compatible with $v$ if and only if $u \dot\wedge (e - v), (e - u) \dot\wedge v \in V^+$ with $u \dot\wedge (e - v) \perp (e - u) \dot\wedge v$. 
\end{proof}

\begin{remark}\label{17a}
Let $u, v \in [0, e]$. Since $(e - u) \dot\wedge (e - v) = e - (u \dot\vee v)$, it follows from Proposition \ref{16} that $u$ is absolutely compatible with $v$ 
if and only if $u \dot\wedge v, u \dot\vee v \in [0, e]$ with $u \dot\wedge v \perp (e - u \dot\vee v)$. Thus by Proposition \ref{17}, $u$ is absolutely compatible 
with $v$ if and only if $u \dot\wedge v$ is absolutely compatible with $u \dot\vee v$.  In particular, $u$ is absolutely compatible to itself if and only if $u \perp (e - u)$.
\end{remark}
Let $A$ be a C$^{\ast}$-algebra. For $a, b \in A$ we shall write, $a \circ b := \frac{1}{2} (a b + b a)$.
\begin{proposition} \label{11}
Let $A$ be a unital C$^{\ast}$-algebra. Then for $a, b \in A_{sa}$, we have $\vert a - b \vert + \vert a + b - 1 \vert = 1$ if and only if 
$a, b \in [0, 1]$ with $a \circ b = a \dot\wedge b$.
\end{proposition}
\begin{proof}
First let $\vert a - b \vert + \vert a + b - 1 \vert = 1$. Note that $\vert x \vert \pm x \in A^+$ for all $x \in A_{sa}$. Thus 
$1 \pm (a - b) \pm (a + b - 1) \in A^+$ so that 
$a, b \in [0, 1]$. Now 
$$\vert a + b - 1 \vert^2 = a^2 + b^2 + 1 + a b + b a - 2 a - 2 b$$
and 
$$(1 - \vert a - b \vert )^2 = 1 + a^2 + b^2 - a b - b a - 2 \vert a - b \vert.$$
Also, by assumption, we have $\vert a + b - 1 \vert = 1 - \vert a - b \vert$ so that 
\begin{eqnarray*}
a^2 + b^2 + 1 + a b + b a - 2 a - 2 b &=& \vert a + b - 1 \vert^2 \\ 
&=& (1 - \vert a - b \vert )^2 \\ 
&=& 1 + a^2 + b^2 - a b - b a - 2 \vert a - b \vert
\end{eqnarray*}
Thus $2 (a b + b a) = 2 (a + b - \vert a - b \vert)$ so that $a \circ b = a \dot\wedge b$.

Conversely, let $a, b \in [0, 1]$ with $a \circ b = a \dot\wedge b$. Expanding $a \circ b = a \dot\wedge b$ as above, we can show that 
$\vert a + b - 1 \vert^2 = (1 - \vert a - b \vert )^2$. Now, as $a, b \in [0, 1]$, we have 
$$- 1 \le - b \le a - b \le a \le 1$$
so that $\vert a - b \vert \le 1$. Thus $\vert a + b - 1 \vert = 1 - \vert a - b \vert$.
\end{proof} 
\begin{corollary}
Let $A$ be a unital C$^{\ast}$-algebra and $a, b \in A^+$. Set $\alpha = \max \{ \Vert a \Vert, \Vert b \Vert \}$. Then 
 $a$ is absolutely comparable with $b$ if and only if $a \circ b = \alpha (a \dot\wedge b)$.
\end{corollary}
For a projection, Proposition \ref{11} takes the following form. 
\begin{proposition}\label{12}
Let $A$ be a unital C$^{\ast}$-algebra. Then for $0 \le a \le 1$ in $A$ and for a projection $p$ in $A$, the following statements are equivalent:
\begin{enumerate}
\item $a$ is absolutely comparable with $p$; 
\item $a \dot\wedge p = a p$;
\item $a p = p a$.
\end{enumerate} 
In this case, $\inf \{ a, p \}$ exists (in $A^+$)  and is equal to $a \dot\wedge p$. 
\end{proposition}
\begin{proof}
$(1) \implies (2)$: 
First, assume that $\vert a - p \vert + \vert a + p - 1 \vert = 1$. Let $a - p = x_1 - x_2$ and $a + p - 1 = y_1 - y_2$ 
such that $x_1, x_2, y_1, y_2 \in A^+$ with $x_1 x_2 = 0$ and $y_1 y_2 = 0$. Then $\vert a - p \vert = x_1 + x_2$ and $\vert a + p - 1 \vert = y_1 + y_2$. Thus 
$x_1 + x_2 + y_1 + y_2 = 1$ and $x_1 - x_2 + y_1 - y_2 =  2 a - 1$ so that $a = x_1 + y_1$ and $p = x_2 + y_1$ and consequently, $1 - p = x_1 + y_2$. 
As $p$ and $1 - p$ are projections with $p (1 - p) = 0$, it follows that $x_1 y_1 = 0$ and $x_2 y _2 = 0$. In particular, $x_1 p = 0$ so that $a p = y_1$. Thus 
$$a \dot\wedge p = \frac{1}{2} \{ a + p - \vert a - p \vert \} = y_1 = a p.$$
As $a \dot\wedge p$ is self-adjoint, $(2) \implies (3)$ is evident.

$(3) \implies (1)$: Next, assume that $a p = p a$. Then $a p = p a p$ and $a (1 - p) = (1 - p) a (1 - p)$ so that $a = p a p + (1 - p) a (1 - p)$. As $0 \le a \le 1$, 
we see that $0 \le p a p \le p$ and $0 \le (1 - p) a (1 - p) \le (1 - p)$. Thus 
\begin{eqnarray*}
\vert a - p \vert &=& \vert (1 - p) a (1 - p) - (p - p a p) \vert \\ 
&=& (1 - p) a (1 - p) + (p - p a p)
\end{eqnarray*}
and 
\begin{eqnarray*}
\vert a + p - 1 \vert &=& \left\vert \left( (1 - p) - (1 - p) a (1 - p) \right) - p a p \right\vert \\ 
&=& \left( (1 - p) - (1 - p) a (1 - p) \right) + p a p.
\end{eqnarray*}
Adding them, we get $\vert a - p \vert + \vert a + p - 1 \vert = 1$. 

Finally, assume that $a p = p a$. Then $a p = a^{\frac{1}{2}} p a^{\frac{1}{2}} \le a$ and $ap = p a p \le p$. As $a, p \in A^+$ and $a p = p a$, we have 
$a p \in A^+$. Next, let $x \in A^+$ be such that $x \le a$ and $x \le p$. As $p$ is a projection, we get $x p = p x = x$. Thus 
$$x = p x p \le p a p = a p$$
so that $a \dot\wedge p = a p = \inf \{ a, p \}$.
\end{proof} 
\begin{corollary}
Let $A$ be a unital C$^{\ast}$-algebra. Then for $a \in A^+$ and for a projection $p$ in $A$, the following statements are equivalent:
\begin{enumerate}
\item $a$ is absolutely comparable with $p$; 
\item $a \dot\wedge p = a p$;
\item $a p = p a$.
\end{enumerate}
\end{corollary}
Commuting projections yield the following refinement of Proposition \ref{12}. 
\begin{proposition}\label{13}
Let $A$ be a unital C$^{\ast}$-algebra and let $\mathcal{P}(A)$ denote the set of projections in $A$. Then for $p, q \in \mathcal{P}(A)$, 
we have $p q = q p$ if and only if $p \dot\wedge q \in \mathcal{P}(A)$. In this case, $p \dot\wedge q = p q = \inf_{\mathcal{P}(A)} \{p, q \}$.
\end{proposition}
\begin{proof}
First let $p q = q p := r$. Then $r = \inf_{\mathcal{P}(A)} \{p, q \}$. We show that $r = p \dot\wedge q$. Since $p - r, q - r \in \mathcal{P}(A)$ 
with $(p - r) (q - r) = 0$, we get 
$$\vert p - q \vert = \vert (p - r) - (q - r) \vert = (p - r) + (q - r).$$
Thus $p \dot\wedge q = \frac{1}{2} \{ p + q - \vert p - q \vert \} = r$.

Conversely, assume that $p \dot\wedge q \in \mathcal{P}(A)$. Since $p \dot\wedge q \le p$ and $p \dot\wedge q \le q$, we get
$p (p \dot\wedge q) = (p \dot\wedge q) p = p \dot\wedge q$ and $q (p \dot\wedge q) = (p \dot\wedge q) q = p \dot\wedge q$. Also, then 
$$\vert (p - p \dot\wedge q) - (q - p \dot\wedge q) \vert = \vert p - q \vert = p + q - 2 (p \dot\wedge q),$$
so that $(p - p \dot\wedge q) (q - p \dot\wedge q ) = 0 = (q - p \dot\wedge q) (p - p \dot\wedge q )$. Thus $p q = q p = p \dot\wedge q$.
\end{proof} 

\section{Commuting projections}

Now we shall present order theoretic replicas of Propositions \ref{12} and \ref{13}. 
\begin{definition}
Let $V$ be an ordered vector space. For $u \in V^+$ we set 
$$V_u := \{ v \in V: k u \pm v \in V^+ ~\textrm{ for some} ~ k > 0 \}.$$
If $(V, e)$ is an order unit space, then $u \in V^+$ is said to have the {\it order unit property} in $V$, if for any $v \in V_u$ we have, 
$\pm v \le \Vert v \Vert u$. In this case, $(V_u, u)$ is also an order unit space and $\Vert v \Vert_u = \Vert v \Vert_e$ for each $v \in V_u$. 

If $(V, e)$ is an absolute order unit space, then $u \in V^+$ is said to have the {\it absolute order unit property} in $V$, if for any 
$v \in V_u$ we have, $\vert v \vert \le \Vert v \Vert u$. In this case, $(V_u, u)$ is also an absolute order unit space and 
$\Vert v \Vert_u = \Vert v \Vert_e$ for each $v \in V_u$.
\end{definition}
\begin{definition}
Let $(V, e, \vert\cdot\vert )$ is an absolute order unit space. Consider the set 
$$OP(V) = \{ p \in [0, 1]: p \perp (e - p) \}.$$
Note that $0, e \in OP(V)$ and that $e - p \in OP(V)$ if $p \in OP(V)$. We shall write $p^{\prime}$ for $e - p \in OP(V)$. 
Elements of $OP(V)$ will be called {\it order projections} for the following reason. 
\end{definition}
\begin{theorem} \label{18}
Let $A$ be a unital C$^{\ast}$-algebra. For $a \in [0, 1]$, these statements are equivalent:
\begin{enumerate}
\item $a$ is a projection in $A$;
\item $a$ is an extreme point of $[0, 1]$;
\item $a$ has the order unit property in $A$; 
\item $[0, a] \cap [0, 1 - a] = \{ 0 \}$;
\item $a \perp (1 - a)$.
\end{enumerate}
\end{theorem}
\begin{proof}
Whereas the equivalence of (1) and (2) is a classical result of C$^{\ast}$-algebra theory, the 
equivalence of (1) and (3) was observed in \cite[Corollary 3.2 and Theorem 3.3]{AK0}. Further, as 
$a$ and $1 - a$ commute, the equivalence of (4) and (5) follows from \cite[Proposition 4.1]{AK2}. 

(1) implies (4): Assume that $a^2 = a$ and let $x \in [0, a] \cap [0, 1 - a]$. Then $0 \le x \le a$ and 
$0 \le x \le 1 - a$. Thus 
$$0 \le (1 - a) x (1 - a) \le (1 - a) a (1 - a) = 0$$
so that $(1 - a) x (1 - a) = 0$ and consequently, $x = a x = x a = a x a$. Now as $0 \le x \le 1 - a$, 
as above we get $a x a = 0$ so that $x = 0$.

(4) implies (1): Finally, assume that $[0, a] \cap [0, 1 - a] = \{ 0 \}$. As $0 \le a \le 1$, we have 
$0 \le a^2 \le a$ and consequently, $0 \le a - a^2 \le a$. Also, 
$$a - a^2 = (1 - a)^{\frac{1}{2}} a (1 - a)^{\frac{1}{2}} \le (1 - a)$$
so that by assumption, $a - a^2 = 0$. Thus $a$ is a projection.
\end{proof}
\begin{remark} 
Note that $a$ is a projection if and only if $1 - a$ is also a projection. Thus we can replace $a$ by $1 - a$ 
 in (1), (2) and (3) of Theorem \ref{18}.
\end{remark} 

Throughout the section, $V$ will be an absolute order unit space and $OP(V)$ will denote the set of all order projections in $V$. 
\begin{proposition}\label{26}
For $p, q \in OP(V)$, the following statements are equivalent:
\begin{enumerate}
\item $p + q \le e$;
\item $p \perp q$;
\item $p + q \in OP(V)$, and 
\item $p \perp_{\infty} q$.
\end{enumerate}
\end{proposition}
\begin{proof}
(1) implies (2): Let $p + q \le e$. Then $0 \le p \le e - q$ so that $p \perp q$, by the definition.

(2) implies (4): By Proposition \ref{6}.

(4) imlies (1): Let $p \perp_{\infty} q$. Then $\Vert p + q \Vert = \max \{ \Vert p \Vert, \vert q \Vert \} \le 1$ so that $p + q \le e$.

(1) implies (3): Let $p + q \le e$ so that $p \le e - q$. As $p, q \in OP(V)$, we have $p \perp e - p$ and $q \perp e - q$. Now, 
$0 \le e - (p + q) \le e - p$ and $e - (p + q) \le e - q$. Thus $e - (p + q) \perp p$ and $e - (p + q) \perp q$. It follows from the additivity 
of $\perp$ that $e - (p + q) \perp p + q$ so that $p + q \in OP(V)$. Finally, (3) implies (1) by the definition of $OP(V)$.
\end{proof}
\begin{proposition}\label{27}
Let $u, v \in [0, e]$. If $u + v \in OP(V)$ with $u \perp v$, then $u, v \in OP(V)$.
\end{proposition}
\begin{proof}
Set $p = u + v$ so that $u \perp (p - u)$. Since $p \in OP(V)$, we have $p \perp (e - p)$. Since $u \le p$, we get $u \perp (e - p)$. Now 
by the additivity of $\perp$, we get $u \perp (e - u)$ so that $u \in OP(V)$. Similarly, we can show that $v \in OP(V)$.
\end{proof}
\begin{corollary}\label{28}
Let $p, q \in OP(V)$ such that $p \le q$. Then $q - p \in OP(V)$.
\end{corollary}
\begin{proof}
Set $q - p = r$ so that $r \in [0, e]$. Since $0 \le r \le e - p$ and $p \perp (e - p)$, we have $p \perp r$. Thus by Proposition \ref{27}, 
$r \in OP(V)$.
\end{proof}
\begin{remark}
Let $p, q \in OP(V)$ such that $p \le q$. Then $p \perp (q - p)$.
\end{remark}
\begin{proposition}\label{29}
Let $u, v \in V$. Then 
\begin{enumerate}
\item $u \dot\vee v + u \dot\wedge v = u + v$; 
\item $u \dot\vee v - u \dot\wedge v = \vert u - v \vert$;
\item $u - u \dot\wedge v, v - u \dot\wedge v \in V^+$; and
\item $(u - u \dot\wedge v) \perp (v - u \dot\wedge v)$. 
\end{enumerate}
\end{proposition}
\begin{proof}
The statements (1), (2) and (3) follow from the definitions of $\dot\wedge$ and $\dot\vee$; and (4) follows from the fact that 
$$(u - u \dot\wedge v) \dot\wedge (v - u \dot\wedge v) = u \dot\wedge v - u \dot\wedge v = 0.$$ 
\end{proof}
\begin{theorem}\label{30}
Let $V$ be an absolute order unit space and let $p, q \in OP(V)$. Then the following statements are equivalent:
\begin{enumerate}
\item $p \dot\wedge q \in OP(V)$.
\item $p \dot\vee q \in OP(V)$.
\item $p^{\prime} \dot\wedge q^{\prime} \in OP(V)$.
\item $p^{\prime} \dot\vee q^{\prime} \in OP(V)$.
\end{enumerate}
\end{theorem}
\begin{proof}
(1) implies (2): Let $p \dot\wedge q \in OP(V)$. By Proposition \ref{29}, $p \dot\wedge q \le p$ so that by Corollary \ref{28}, 
$p - (p \dot\wedge q) \in OP(V)$. Now, $(p \dot\wedge q) + (p - p \dot\wedge q) = p \in OP(V)$ so that by Proposition \ref{26}, we get 
$p \dot\wedge q \perp (p - p \dot\wedge q)$. Since $(p - p \dot\wedge q) \perp (q - p \dot\wedge q)$ by Proposition \ref{26}, the 
additivity of $\perp$ yields, $q \perp (p - p \dot\wedge q)$. Again invoking Propositions \ref{26} and \ref{29}, we may conclude that 
$p \dot\vee q = q + p - p \dot\wedge q \in OP(V)$. \\
(2) implies (3): We have $p^{\prime} \dot\wedge q^{\prime} = e - (p \dot\vee q) \in OP(V)$, if $p \dot\vee q \in OP(V)$.\\
Now, (3) implies (4) by step one and (4) implies (1) by step two.
\end{proof}
\begin{theorem}\label{31}
Let $V$ be an absolute order unit space and let $p, q \in OP(V)$. Then the following statements are equivalent:
\begin{enumerate}
\item $p \dot\wedge q \in OP(V)$;
\item $p$ is absolutely compatible with $q$;
\item $p \dot\wedge q, p \dot\wedge q^{\prime} \in OP(V)$.
\end{enumerate}
\end{theorem}
\begin{proof}
Note that 
(3) is stronger than (1).\\ 
(1) implies (2): Let $p \dot\wedge q \in OP(V)$. Then by Theorem \ref{30}, $p^{\prime} \dot\wedge q^{\prime} \in OP(V)$. Now 
\begin{align*}
p \dot\wedge q + p^{\prime} \dot\wedge q^{\prime} &= \frac{1}{2} \{ p + q - \vert p - q \vert + p^{\prime} + q^{\prime} - \vert p^{\prime} - q^{\prime} \vert \} \\
&= \frac{1}{2} \{ 2 e - 2 \vert p - q \vert \} \\ 
&= e - \vert p - q \vert 
\end{align*}
for $p + p^{\prime} = e, q + q^{\prime} = e$ and $p^{\prime} - q^{\prime} = e - p - e + q = q - p$. Similarly, 
\begin{align*}
p \dot\wedge q - p^{\prime} \dot\wedge q^{\prime} &= \frac{1}{2} \{ p + q - \vert p - q \vert - p^{\prime} - q^{\prime} + \vert p^{\prime} - q^{\prime} \vert \} \\
&= \frac{1}{2} \{ 2 p + 2 q - 2 e \} \\ 
&= p + q - e.
\end{align*}
Since $p \dot\wedge q \le p$ and $p^{\prime} \dot\wedge q^{\prime} \le p^{\prime}$, and since $p \perp p^{\prime}$, we have $p \dot\wedge q \perp p^{\prime} \dot\wedge q^{\prime}$. 
Thus $\vert p \dot\wedge q - p^{\prime} \dot\wedge q^{\prime} \vert = p \dot\wedge q + p^{\prime} \dot\wedge q^{\prime}$ so that $\vert p + q - e \vert = e - \vert p - q \vert$. \\ 
(2) implies (3): Let $\vert p - q \vert + \vert p - q^{\prime} \vert = e$. Then 
\begin{eqnarray*}
p \dot\wedge q &=& \frac{1}{2} \{ p + q - \vert p - q \vert \} \\
&=& \frac{1}{2} \{ p + q  - e + \vert p - q^{\prime} \vert \} \\
&=& \frac{1}{2} \{ p - q^{\prime} + \vert p - q^{\prime} \vert \} \in V^+.
\end{eqnarray*}
Similarly, $p \dot\wedge q^{\prime} \in V^+$. Further, as $p \dot\wedge q \le q$, $p \dot\wedge q^{\prime} \le q^{\prime}$ and $q \perp q^{\prime}$, we get that 
$p \dot\wedge q \perp p \dot\wedge q^{\prime}$. Thus by Proposition \ref{27}, $p \dot\wedge q, p \dot\wedge q^{\prime} \in OP(V)$ for 
\begin{eqnarray*}
p \dot\wedge q + p \dot\wedge q^{\prime} &=& \frac{1}{2} \{ p + q  - \vert p - q \vert + p + q^{\prime} - \vert p - q^{\prime} \vert \} \\
&=& \frac{1}{2} \{ 2 p +  e  - e \} = p \in OP(V).
\end{eqnarray*}
\end{proof}
\begin{remark}\label{32}
Let $p, q \in OP(V)$ such that $p \dot\wedge q \in OP(V)$. Then $r \dot\wedge s, r \dot\vee s \in OP(V)$ whenever $r, s \in \{ p, q, p^{\prime}, q^{\prime} \}$.
\end{remark}
\begin{proposition}\label{33}
Let $V$ be an absolute order unit space and let $p, q \in OP(V)$. Then $p \dot\wedge q \in OP(V)$ if and only if $\vert p - q \vert \in OP(V)$ with $\vert p - q \vert \le p + q$. 
\end{proposition}
\begin{proof}
First, let $p \dot\wedge q \in OP(V)$. Then $p \dot\vee q \in OP(V)$ by Theorem \ref{30}. Also, $p \dot\wedge q \le p\dot\vee q$ with $\vert p - q \vert = p \dot\vee q - p \dot\wedge q$ 
by Proposition \ref{29}. Thus $\vert p - q \vert \in OP(V)$ by Corollary \ref{28}. Further, as $p \dot\wedge q \ge 0$, we have $\vert p - q \vert \le p + q$. \\
Conversely, let $\vert p - q \vert \in OP(V)$ with $\vert p - q \vert \le p + q$. Then $p \dot\wedge q \ge 0$ so that $p \dot\wedge q \in [0, e]$ for $p \dot\wedge q \le p \le e$. Next, $(p - p \dot\wedge q) + (q - p \dot\wedge q) = \vert p - q \vert \in OP(V)$. Also, by Proposition \ref{29}, we have 
$(p - p \dot\wedge q) \perp (q - p \dot\wedge q)$. Thus by Proposition \ref{27}, we get that $(p - p \dot\wedge q) \in OP(V)$. Since $p \dot\wedge q \ge 0$, 
we have $(p - p \dot\wedge q) \le p$. Thus by Corollary \ref{28}, we have $p \dot\wedge q = p - (p - p \dot\wedge q) \in OP(V)$.
\end{proof}
\begin{theorem}\label{34}
Let $V$ be an absolute order unit space and let $p, q \in OP(V)$ with $p \dot\wedge q \in OP(V)$. Then $\inf_{OP(V)} \{ u, v \}$ exists and is equal to 
$p \dot\wedge q$.
\end{theorem}
\begin{proof}
Let $r \le p$ and $r \le q$ for some $r \in OP(V)$. Then $r \perp p^{\prime}$ and $r \perp q^{\prime}$. Thus by the additivity of $\perp$, we get 
$r \perp (p^{\prime} + q^{\prime})$. Since $p \dot\wedge q \in OP(V)$, we have $p^{\prime} \dot\wedge q^{\prime}, p^{\prime} \dot\vee q^{\prime} \in OP(V)$ 
by Theorem \ref{30}. Thus using Proposition \ref{29}, we see that $0 \le p^{\prime} \dot\vee q^{\prime} = p^{\prime} + q^{\prime} - (p^{\prime} \dot\wedge 
q^{\prime}) \le p^{\prime} + q^{\prime}$ so that $r \perp (p^{\prime} \dot\vee q^{\prime})$. As $r, p^{\prime} \dot\vee q^{\prime} \in OP(V)$, by Proposition 
\ref{26}, we have $r + p^{\prime} \dot\vee q^{\prime} \le e$. But $p^{\prime} \dot\vee q^{\prime} = e - (p \dot\wedge q)$ so that $r \le p \dot\wedge q$. 
Now as $p \dot\wedge q \le p, q$, we conclude that $\inf_{OP(V)} \{ u, v \}$ exists and is equal to $p \dot\wedge q$.
\end{proof}
\begin{corollary}\label{35}
Let $V$ be an absolute order unit space and let $p, q \in OP(V)$ with $p \dot\vee q \in OP(V)$. Then $\sup_{OP(V)} \{ p, q \}$ exists and is equal to 
$p \dot\vee q$.
\end{corollary}
\begin{proposition}
Let $V$ be an absolute order unit space and let $p, q \in OP(V)$ with $p \dot\wedge q \in OP(V)$ 
so that $r = \vert p - q \vert \in OP(V)$. Then the set $S := \{ 0, e, p, q, r, p^{\prime}, q^{\prime}, 
r^{\prime} \}$ is closed under the binary operation $(u, v) \mapsto \vert u - v \vert$.
\end{proposition}
\begin{proof}
Since $p \dot\wedge q \in OP(V)$, we have that $\vert e - p - q \vert = e - \vert p - q \vert = e - r$ and 
that $p \dot\wedge q^{\prime}, p^{\prime} \dot\wedge q, p^{\prime} \dot\wedge q^{\prime} \in OP(V)$. 
Also $p \dot\wedge q + p \dot\wedge q^{\prime} = p$. Since $p \dot\wedge q \le q$ and since 
$p \dot\wedge q^{\prime} \le q^{\prime}$, we have $p \dot\wedge q \perp p \dot\wedge q^{\prime}$ 
so that $\vert p \dot\wedge q - p \dot\wedge q^{\prime} \vert = p \dot\wedge q + 
p \dot\wedge q^{\prime} = p$. also 
$$2 ( p \dot\wedge q - p \dot\wedge q^{\prime} ) = (p + q - r) - (p + q^{\prime} - r^{\prime} ) = 2 (q - r)$$ 
so that $p = \vert q - r \vert$. Now, by symmetry we can get $q = \vert p - r \vert$. In a similar way, 
we can calculate $\vert u - v \vert$ for any $u, v \in S$ to complete the proof.
\end{proof}

\section{Expanding the scope}
In this section, we shall examine absolute compatibility of an order projection with a general positive element. First we note that 
order projections in an absolute order unit space have the following `norming' property. 

\begin{proposition}\label{19}
Let $V$ be an absolute order unit space.
\begin{enumerate}
\item If $p \in OP(V)$, then $p$ has the absolute order unit property in $V$. Dually, $e - p$ also has the absolute order unit property in $V$.
\item Let $u \in [0, e]$ and assume that $u$ and $e - u$ have the absolute order unit property in $V$. Then $u \perp_{\infty}^a (e - u)$. 
\end{enumerate}
\end{proposition}
\begin{proof}
(1): First, assume that $p \in OP(V)$. Let $k p \pm v \in V^+$ for some $k > 0$. Set $v_1 = \frac{1}{2} (k p + v)$ and $v_2 = \frac{1}{2} (k p - v)$. Then 
$v_1, v_2 \in V^+$ with $v_1 - v_2 = v$ and $v_1 + v_2 = k p$. As $p \in OP(V)$, we have $p \perp (e - p)$ and consequently, $v_1 \perp (e - p)$ and 
$v_2 \perp (e - p)$. Thus $\vert v \vert = \vert v_1 - v_2 \vert \perp (e - p)$. Since $V$ is an absolute order unit space, we further have $\vert v \vert \perp_{\infty}^a 
(e - p)$ so that $\Vert \Vert \vert v \vert \Vert^{-1} \vert v \vert + (e - p) \Vert = 1$. Now, for $f \in S(V)$, we have  
$$1 = \left\Vert \Vert v \Vert^{-1} \vert v \vert + (e - p) \right\Vert \ge f \left( \Vert v \Vert^{-1} \vert v \vert + (e - p) \right) = 1 - f(p - \Vert v \Vert^{-1} \vert v \vert ).$$
Thus $f(p - \Vert v \Vert^{-1} \vert v \vert ) \ge 0$ for all $f \in S(V)$ so that $\vert v \vert \le \Vert v \Vert p$. In other words, $p$ has the absolute order unit property 
in $V$.

(2): Conversely, let $u \in [0, e]$ and assume that $u$ and $e - u$ have the (absolute) order unit property in $V$. Let $0 \le v \le u$ and $0 \le w \le e - u$. 
By the order unit property of $u$ and $e - u$ we have $v \le \Vert v \Vert u$ and $w \le \Vert w \Vert (e - u)$. Thus 
$$0 \le \Vert v \Vert^{-1} v + \Vert w \Vert^{-1} w \le u + (e - u) = e$$
so that $\Vert \Vert v \Vert^{-1} v + \Vert w \Vert^{-1} w \Vert \le 1$. Now 
$$1 = \left\Vert \Vert v \Vert^{-1} v \right\Vert \le \left\Vert \Vert v \Vert^{-1} v + \Vert w \Vert^{-1} w \right\Vert \le 1$$
so that $u \perp_{\infty}^a (e - u)$.
\end{proof}
\begin{remark}\label{20} 
\begin{enumerate}
\item Let an absolute order unit space $(V, e)$ satisfy $(O. \perp_{\infty}.2)$. Then $p \in OP(V)$ if and only if $p \in [0, e]$ and both $p$ and $e - p$ satisfy the 
absolute order unit property. 
\item Let $(V, e)$ be an order unit space and let $u \in V^+$ has the order unit property in $V$. Then $(V_u, u)$ is an order unit space. In this case, 
$u$ and $e$ determine the same norm in $V_u$.Thus $(V_u, u)$ is an `order unit ideal' of $V$.
\item Let $(V, e)$ be an absolute order unit space and assume that $u \in V^+$ has the absolute order unit property in $V$. Then $(V_u, u)$ is an 
absolute order unit space. In particular, if $p \in OP(V)$, then $(V_p, p)$ is an absolute order unit space and in this case, 
$$OP(V_p) = \{ q \in OP(V): q \le p \} = OP(V) \cap V_p.$$
\end{enumerate}
\end{remark}
\begin{proposition}\label{18a}
Let $(V, e)$ be an order unit space and let $u \in [0, e]$. If $u$ has the order unit property in $V$, then 
it is an extreme point in $[0, e]$.
\end{proposition}
\begin{proof}
Let $u = \alpha v + (1 - \alpha ) w$ for some $v, w \in [0, e]$ and $0 < \alpha < 1$. Then $0 \le \alpha v 
\le u$. Since $u$ has the order unit property in $V$, we get, $v \le \Vert v \Vert u \le u$. Similarly, 
$0 \le w \le u$. Set $u - v = v_1$ and $u - w = w_1$. Then 
$$0 = u - (\alpha v + (1 - \alpha ) w) = \alpha v_1 + (1 - \alpha ) w_1.$$
Since $V^+$ is proper and $0 < \alpha < 1$, we get $v_1 = 0 = w_1$ so that $v = u = w$.
\end{proof}

The absolute compatibility between an order projection $p \in OP(V)$ and an arbitrary element $u \in [0, e]$ is related to 
$V_p + V_{p^{\prime}}$ which we describe below. 
Let $(V_1, e_1)$ and $(V_2, e_2)$ be any two absolute order unit spaces. Consider
$$V = V_1 \times V_2; \quad V^+ = V_1^+ \times V_2^+; ~\textrm{and}~ e = (e_1, e_2).$$
Then $(V, e)$ becomes an absolute order unit space in a canonical way. Further, $V$ is unitally and 
isometrically order isomorphic to $V_1 \oplus_{\infty} V_2$.
\begin{proposition}\label{21}
Let $(V, e)$ be an absolute order unit normed space and let $p, q \in OP(V)$ with $p + q \le e$. Then 
$V_p + V_q$ is an absolute order smooth $\infty$-normed subspace of $(V, e)$ and is isometrically 
order isomorphic to $V_p \oplus_{\infty} V_q$. In particular,for $x \in V_p$ and $y \in V_q$, we have 
$(x + y)^+ = x^+ + y^+$ and $(x + y)^- = x^- + y^-$ so that $\vert x + y \vert = \vert x \vert + \vert y \vert$.
\end{proposition}
\begin{proof}
Let $u \in V_p$ and $v \in V_q$. Then $u^+, u^- \in V_p^+$ and $v^+, v^- \in V_q^+$. As $p + q \le e$, 
we have $p \perp_{\infty}^a q$ and consequently, we may conclude that $u^+, u^-, v^+$ and $v^-$ are 
absolutely $\infty$-orthogonal to each other. Now, by the additivity, we have $(u^+ + v^+) 
\perp_{\infty}^a (u^- + v^-)$. Thus as $u + v = (u^+ + v^+) - (u^- + v^-)$, we get that $(u + v)^+ = u^+ + 
v^+; (u + v)^- =  u^- + v^-$ so that $\vert u + v \vert = \vert u \vert + \vert v \vert$. Therefore, $V_p + V_q$ 
is an absolute ordered space. It also follows that $V_p \cap V_q = \{ 0 \}$. Finally, for $u \in V_p$ and 
$v \in V_q$ we note that
\begin{eqnarray*}
\Vert u + v \Vert &=& \max \{ \Vert (u + v)^+ \Vert, \Vert (u + v)^- \Vert \} \\
&=& \max \{ \Vert u^+ + v^+ \Vert, \Vert u^- + v^- \Vert \} \\
&=& \max \{ \Vert u^+ \Vert, \Vert v^+ \Vert, \Vert u^- \Vert, \Vert v^- \Vert \} \\
&=& \max \{ \Vert u \Vert, \Vert v \Vert \} 
\end{eqnarray*}
which completes the proof.
\end{proof}

\begin{theorem}\label{23}
Let $(V, e)$ be an absolute order unit normed space and let 
$p, q \in OP(V)$ with $p + q \le e$. Then for $w \in V$, the following statements are equivalent:
\begin{enumerate}
\item $w \in V_p^+ + V_q^+$ with $\Vert w \Vert \le 1$;
\item $w = p \dot\wedge w + q \dot\wedge w$;
\item $\vert p - w \vert + \vert q - w \vert = p + q$.
\end{enumerate}
\end{theorem}
\begin{proof}
(1) implies (2): Let $w \in V_p^+ + V_q^+$ with $\Vert w \Vert \le 1$. By Proposition \ref{21}, there exists 
a unique pair $w_p \in V_p^+$ and $w_q \in V_q^+$ such that $w = w_p + w_q$. As $\Vert w \Vert \le 1$, 
we get that $w_p \le p$ and $w_q \le q$ so that $0 \le p - w_p \le p$. Since $p + q \le e$, we have 
$p \perp_{\infty}^a q$ and consequently, $(p - w_p) \perp_{\infty}^a w_q$. Thus
$$\vert p - w \vert = \vert p - w_p - w_q \vert = p - w_p + w_q.$$
Therefore,
$$u \dot\wedge w = \frac{1}{2} \left( p + w - \vert p - w \vert \right) = w_p.$$
Similarly, $q \dot\wedge w = w_q$ so that $w = p \dot\wedge w + q \dot\wedge w$.

(2) implies (3): Let $w = p \dot\wedge w + q \dot\wedge w$. Then 
$$2 w = p + w - \vert p - w \vert + q + w - \vert q - w \vert$$ 
so that $\vert p - w \vert + \vert q - w \vert = p + q$.

(3) implies (1): Finally, let $\vert p - w \vert + \vert q - w \vert = p + q$. Consider the $\perp_{\infty}^a$- 
decompositions $p - w = u_1 - u_2$ and $q - w = v_1 - v_2$ in $V^+$. Then $\vert p - w \vert = u_1 + 
u_2$ and $\vert q - w \vert = v_1 + v_2$. Now it follows that 
$$u_1 + u_2 + v_1 + v_2 = p + q$$ 
and 
$$u_1 - u_2 + v_1 - v_2 = p + q - 2 w.$$ 
Therefore, $w = u_2 + v_2$ and consequently, $p = u_1 + v_2$ and $q = v_1 + u_2$. Now, as 
$0 \le v_2 \le p$ and $0 \le u_2 \le q$, we have $w \in V_p^+ + V_q^+$. Further as $p \perp_{\infty}^a 
q$, we have $u_2 \perp_{\infty}^a v_2$. Thus 
$$\Vert w \Vert = \Vert u_2 + v_2 \Vert = \max \{ \Vert u_2 \Vert, \Vert v_2 \Vert \} \le \max \{ \Vert p 
\Vert, \Vert q \Vert \} \le 1.$$
\end{proof} 
\begin{remark}
It follows from the proof of Theorem \ref{23} that if $p, q \in OP(V)$ with $p + q \le e$ and if 
$w \in V_p^+ + V_q^+$ with $\Vert w \Vert \le 1$, then the $p$ and $q$ ``components'' of $w$ are 
$p \dot\wedge w$ and $q \dot\wedge w$ respectively. More generally, if $w \in V_p^+ + V_q^+$, then 
the $p$ and $q$ ``components'' of $w$ are $(k p) \dot\wedge w = (\Vert w \Vert p) \dot\wedge w$ and 
$(k q) \dot\wedge w = (\Vert w \Vert q) \dot\wedge w$ respectively for any $k \ge \Vert w \Vert$. 
\end{remark} 
We shall write $AC(p) := V_p + V_{p^{\prime}}$ so that $AC(p)^+ := (V_p + V_{p^{\prime}}) \cap V^+ = V_p^+ + V_{p^{\prime}}^+$. 

\begin{corollary}\label{24}
Let $(V, e)$ be an absolute order unit normed space, $p \in OP(V)$ and $u \in [0, e]$. Then $u$ is absolutely compatible with $p$ if and only if 
$u \in AC(p)^+$. 
\end{corollary}
\begin{proposition}\label{22}
Let $(V, e)$ be an absolute order unit normed space, $p \in OP(V)$ and $u \in [0, e]$ and assume that 
$u$ is absolutely compatible with $p$. Then 
\begin{enumerate}
\item $\inf \{ u, p \}$ exists in $[0, e]$ and is equal to $u \dot\wedge p$; and
\item $\sup \{ u, p \}$ exists in $[0, e]$ and is equal to $u \dot\vee p$.
\end{enumerate}
\end{proposition}
\begin{proof}
(1): By the definition, we have $u \dot\wedge p \le u$ and $u \dot\wedge p \le p$. As $u$ is absolutely compatible 
with $p$, we have $u \dot\wedge p, u \dot\wedge p^{\prime} \in V^+$ with $u = u \dot\wedge p + u \dot\wedge p^{\prime}$. 
Next, let $w \in V^+$ such that $w \le u$ and $w \le p$. Then $w \in V_p^+$. Also, $u \dot\wedge p \in V_p^+$ 
so that $u \dot\wedge p - w \in V_p$. Further $u \dot\wedge p^{\prime} \in V_{p^{\prime}}^+$.  Thus by Theorem \ref{23}, 
we get 
\begin{eqnarray*}
u - w = \vert u - w \vert &=& \vert (u \dot\wedge p - w) + (u \dot\wedge p^{\prime}) \vert \\
&=& \vert u \dot\wedge p - w \vert + \vert u \dot\wedge p^{\prime} \vert \\
&=& \vert u \dot\wedge p - w \vert +  u \dot\wedge p^{\prime} 
\end{eqnarray*} 
so that $u \dot\wedge p - w = \vert u \dot\wedge p - w \vert \in V^+$. Hence $u \dot\wedge p = \inf \{ u, p \}$. 

(2): As $u$ is absolutely compatible with $p$, we have $e - u$ is absolutely compatible with $e - p$. Thus as in (1), we may 
conclude that $(e - u) \dot\wedge (e - p) = \inf \{ e - u, e - p \}$. Now it follows that 
$$u \dot\vee p = e - \left( (e -u)\dot\wedge (e - v)\right)  = e - \inf \{ e - u, e - v \} = \sup \{ u, p \}.$$ 
\end{proof} 

\begin{theorem}\label{25}
Let $(V, e)$ be an absolute order unit normed space and let $p \in OP(V)$. Then for $u_1, \dots, u_n \in AC(p)^+$, we have 
$$(1) \quad \left( \sum_{i=1}^n u_i \right) \dot\wedge \left( \sum_{i=1}^n \Vert u _i \Vert p \right) = \sum_{i=1}^n (u_i \dot\wedge \Vert u_i \Vert p)$$
and 
$$(2) \quad \left( \sum_{i=1}^n u_i \right) \dot\vee \left( \sum_{i=1}^n \Vert u _i \Vert p \right) = \sum_{i=1}^n (u_i \dot\vee \Vert u_i \Vert p).$$
\end{theorem}
\begin{proof}
First, let $u_1, u_2 \in AC(p)^+$ with $\Vert u_i \Vert \le 1, i = 1, 2$ and assume that $\alpha \in [0, 1]$. Then by Corollary \ref{24}, 
we have $u_i = u_i \dot\wedge p + u_i \dot\wedge p^{\prime}$ with $u_i \dot\wedge p, u_i \dot\wedge p^{\prime} \in V^+$ for 
$i = 1, 2$. Thus 
\begin{align*}
\vert p - &\alpha u_1 - (1 - \alpha) u_2 \vert \\
&= \left\vert \{ \alpha (p - u_1 \dot\wedge p) + (1 - \alpha) (p - u_2 \dot\wedge p) \} - 
\{ \alpha (u_1 \dot\wedge p^{\prime}) + (1 - \alpha) (u_2 \dot\wedge p^{\prime}) \} \right\vert \\
&= \{ \alpha (p - u_1 \dot\wedge p) + (1 - \alpha) (p - u_2 \dot\wedge p) \} + 
\{ \alpha (u_1 \dot\wedge p^{\prime}) + (1 - \alpha) (u_2 \dot\wedge p^{\prime}) \}
\end{align*}
for 
$$0 \le  \alpha (p - u_1 \dot\wedge p) + (1 - \alpha) (p - u_2 \dot\wedge p) \le p,$$  
$$0 \le \alpha (u_1 \dot\wedge p^{\prime}) + (1 - \alpha) (u_2 \dot\wedge p^{\prime}) \le p^{\prime}$$
and $p \perp_{\infty}^a p^{\prime}$. On simplifying, we get
$$\vert \alpha u_1 + (1 - \alpha) u_2 - p \vert 
= \alpha u_1 + (1 - \alpha) u_2 + p - 2 \{ \alpha (u_1 \dot\wedge p) + (1 - \alpha) (u_2 \dot\wedge p) \}.$$
Now it follows that 
$$(\alpha u_1 + (1 - \alpha) u_2) \dot\wedge p = \alpha (u_1 \dot\wedge p) + (1 - \alpha) (u_2 \dot\wedge p)$$
and that
$$(\alpha u_1 + (1 - \alpha) u_2) \dot\vee p = \alpha (u_1 \dot\vee p) + (1 - \alpha) (u_2 \dot\vee p).$$
By a standard technique, now we can now show that this fact also holds for $n$ elements: For $u_1, \dots, u_n \in AC(p)^+$ with 
$\Vert u_i \Vert \le 1, i = 1, \dots, n$ and positive real numbers $\alpha_1, \dots, \alpha_n$ with $\sum_{i=1}^n \alpha_i = 1$, we have 
$$(!) \quad \left( \sum_{i=1}^n \alpha_i u_i \right) \dot\wedge p = \sum_{i=1}^n \alpha_i (u_i \dot\wedge p)$$
and 
$$(!!) \quad \left( \sum_{i=1}^n \alpha_i u_i \right) \dot\vee p = \sum_{i=1}^n \alpha_i (u_i \dot\vee p).$$
Finally, let $u_1, \dots, u_n \in AC(p)^+ \setminus \{ 0 \}$. Set $w_i = \Vert u_i \Vert^{-1} u_i$ so that $\Vert w_i \Vert = 1, i = 1, \dots, n$. 
Then for $k = \sum_{i=1}^n \Vert u _i \Vert$, using (!), we get
\begin{align*}
\left( \sum_{i=1}^n u _i \right) \dot\wedge (k p) &= k \left( \left( \sum_{i=1}^n \frac{\Vert u _i \Vert}{k} w_i \right) \dot\wedge p \right) \\
&= k \left( \sum_{i=1}^n \frac{\Vert u _i \Vert}{k} (w_i \dot\wedge p) \right) \\
&= \sum_{i=1}^n \Vert u _i \Vert (w_i \dot\wedge p) \\
&= \sum_{i=1}^n (u_i \dot\wedge \Vert u _i \Vert p)
\end{align*}
which proves (1). Similarly, using (!!), we may get (2). 
\end{proof} 

Let $(V, e)$ be an absolute order unit normed space and let 
$\{ p_i: i \in I \} \subset OP(V)$. We shall write $AC(p_i; i \in I)$ for $\cap_{i \in I} AC(p_i)$.
\begin{theorem}\label{46}
Let $(V, e)$ be an absolute order unit normed space and let $v \in AC^+(p, q)$ for some $p, q \in OP(V)$ such that $p + q \le e$. 
Set $r = e - p - q$ so that $r \in OP(V)$. Then $v \in AC^+(r)$ and $v = v \dot\wedge p + v \dot\wedge q + v \dot\wedge r$. 
\end{theorem}
\begin{proof}
We may assume that $\Vert v \Vert \le 1$. As $v \in AC^+(p)$, we have 

$(i) \qquad v = v \dot\wedge p + v \dot\wedge p^{\prime}.$ \\
Now, $p \le e - q = q^{\prime}$ so that $v \dot\wedge p \in V_p^+ \subset V_{q^{\prime}} \subset AC^+(q)$. Also, $v \in AC^+(q)$ 
so that $v \dot\wedge p^{\prime} = v - v \dot\wedge p \in AC(q)$. But $v \dot\wedge p^{\prime} \in V^+$ so that $v \dot\wedge p^{\prime} 
\in AC^+(q)$. As $\Vert v \dot\wedge p^{\prime} \Vert \le \Vert v \Vert \le 1$ we get 

$(ii) \qquad v \dot\wedge p^{\prime} = (v \dot\wedge p^{\prime}) \dot\wedge q + (v \dot\wedge p^{\prime}) \dot\wedge q^{\prime}.$ \\
We show that $(v \dot\wedge p^{\prime}) \dot\wedge q = v \dot\wedge q$. We have 
\begin{eqnarray*}
\vert v - q \vert &=& \vert v \dot\wedge p + (v \dot\wedge p^{\prime}) \dot\wedge q + (v \dot\wedge p^{\prime}) \dot\wedge q^{\prime} - q \vert \\
&=& \vert \{ q - (v \dot\wedge p^{\prime}) \dot\wedge q \} - \{ v \dot\wedge p + (v \dot\wedge p^{\prime}) \dot\wedge q^{\prime} \} \vert
\end{eqnarray*}
Now, as $0 \le (v \dot\wedge p^{\prime}) \dot\wedge q \le q$, we have $0 \le q - (v \dot\wedge p^{\prime}) \dot\wedge q \le q$. Thus, as 
$0 \le v \dot\wedge p \le p$, $0 \le (v \dot\wedge p^{\prime}) \dot\wedge q^{\prime} \le q^{\prime}$, $q \perp p$ and $q \perp q^{\prime}$; we get 
$\{ q - (v \dot\wedge p^{\prime}) \dot\wedge q^{\prime} \} \perp v \dot\wedge p$ and $\{ q - (v \dot\wedge p^{\prime}) \dot\wedge q^{\prime} \} \perp 
(v \dot\wedge p^{\prime}) \dot\wedge q^{\prime}$. Now, by the additivity of $\perp$, we get $\{ q - (v \dot\wedge p^{\prime}) \dot\wedge q^{\prime} \} \perp 
\{ v \dot\wedge p + (v \dot\wedge p^{\prime}) \dot\wedge q^{\prime} \}$. Thus 
$$\vert v - q \vert = \{ q - (v \dot\wedge p^{\prime}) \dot\wedge q^{\prime} \} + v \dot\wedge p + (v \dot\wedge p^{\prime}) \dot\wedge q^{\prime}$$ 
so that 
\begin{eqnarray*}
2 ( v \dot\wedge q) &=& v + q - \vert v - q \vert \\
&=& v + q - \{ q - (v \dot\wedge p^{\prime}) \dot\wedge q^{\prime} \} - v \dot\wedge p - (v \dot\wedge p^{\prime}) \dot\wedge q^{\prime} \\
&=& 2 \{ (v \dot\wedge p^{\prime}) \dot\wedge q \}.
\end{eqnarray*}
Thus by (i) and (ii), we have 

$(iii) \qquad v = v \dot\wedge p + v \dot\wedge q + (v \dot\wedge p^{\prime}) \dot\wedge q^{\prime}.$\\
Now, interchanging $p$ and $q$, we may conclude that 

$(iv) \qquad (v \dot\wedge p^{\prime}) \dot\wedge q^{\prime} = (v \dot\wedge q^{\prime}) \dot\wedge p^{\prime}$.\\
Finally, we shall show that $(v \dot\wedge p^{\prime}) \dot\wedge q^{\prime} = v \dot\wedge r$. Let us quickly note that $v \dot\wedge p \le p$ and 
$v \dot\wedge q \le q$ so that $v \dot\wedge p + v \dot\wedge q \in V_{p+q}^+$. Also, $(v \dot\wedge p^{\prime}) \dot\wedge q^{\prime} \le p^{\prime}$ 
and $(v \dot\wedge p^{\prime}) \dot\wedge q^{\prime} \le q^{\prime}$ so that $(v \dot\wedge p^{\prime}) \dot\wedge q^{\prime} \le p^{\prime} 
\dot\wedge q^{\prime} = (p + q)^{\prime} = r$. Thus $(v \dot\wedge p^{\prime}) \dot\wedge q^{\prime} \in V_r^+$. Now, it follows from (iii) that 
$v \in AC^+(r)$. Therefore, 

$(v) \qquad \vert v - r \vert + \vert v + r - e \vert = e$. \\
Now, 
\begin{eqnarray*}
\vert v - p - q \vert &=& \vert (p - v \dot\wedge p) + (q - v \dot\wedge q) - (v \dot\wedge p^{\prime}) \dot\wedge q^{\prime} \vert \\
&=& (p - v \dot\wedge p) + (q - v \dot\wedge q) + (v \dot\wedge p^{\prime}) \dot\wedge q^{\prime}
\end{eqnarray*}
for $0 \le (p - v \dot\wedge p) + (q - v \dot\wedge q) \le p + q$ and $0 \le (v \dot\wedge p^{\prime}) \dot\wedge q^{\prime} \le (p + q)^{\prime}$. Thus 
\begin{eqnarray*}
2 ( v \dot\wedge r) &=& v + r - \vert v - r \vert \\
&=& v + e - (p + q) - e + \vert v - (p + q) \vert \quad \textrm{(by (v))} \\
&=& v - (p + q) + (p - v \dot\wedge p) + (q - v \dot\wedge q) + (v \dot\wedge p^{\prime}) \dot\wedge q^{\prime} \\
&=& 2 \{ (v \dot\wedge p^{\prime}) \dot\wedge q^{\prime} \}.
\end{eqnarray*}
Putting it in (iii), we get the result.
\end{proof}
\begin{corollary}\label{45}
Let $(V, e)$ be an absolute order unit normed space and let $w \in AC^+(p, q)$ for some $p, q \in OP(V)$.
\begin{enumerate}
\item If $p + q \le e$, then $w \in AC^+(q + p)$.
\item If $p \le q$, then $w \in AC^+(q - p)$.
\end{enumerate}
\end{corollary}

\begin{remark}\label{47}
Let $(V, e)$ be an absolute order unit normed space and let 
$w \in AC^+(p_i; 1 \le i \le n)$ for some $p_1, \dots, p_n \in OP(V)$. The results of Corollary \ref{45} 
may be generalized in the following way.
\begin{enumerate}
\item If $p = p_1 + \dots + p_n \le e$ so that $p \in AP(V)$ and that $p_1, \dots, p_n, p^{\prime}$ are 
mutually $\perp_{\infty}^a$-orthogonal, we have $w \in V_{p_1}^+ + V_{p_2}^+ + \dots + V_{p_n}^+ 
+ V_{p^{\prime}}^+$ and $w \in AC^+( \sum_{i \in I} p_i)$ whenever $I \subset \{ 1, \dots, n \}$.
\item If $p_1 \le \dots, \le p_n$, then $w \in V_{p_1}^+ + V_{p_2 - p_1}^+ + \dots + V_{p_n - p_{n-1}}^+ 
+ V_{p_n^{\prime}}^+$ and $w \in AC^+(p_{i+k} - p_i)$ whenever $1 \le i, k \le n$ with $i + k \le n$. 
Also, in this case, $p_1, p_2 - p_1, \dots, p_n - p_{n-1}, p_n^{\prime} \in OP(V)$ are mutually 
$\perp_{\infty}^a$-orthogonal.
\end{enumerate}
\end{remark}

\section{Spectral family of order projections}

In this section we shall discuss spectral family of order projections for an element in an 
absolute order unit space $(V, e)$. For this purpose, we need the following concept.

\subsection{A hypothesis for $OP(V)$} \hfill

In general, a C$^{\ast}$-algebra may not have sufficiently many projections. However, in a von 
Neumann algebra $M$, $OP(M)$ always covers $M_{sa}$ in the following sense. 
\begin{definition}
	Let $(V, e)$ be an absolute order unit space. We say that $p \in OP(V)$ covers an element $v \in V$,  
	if $v \in V_p$ and $V_p \subset V_q$ whenever $v \in V_q$.
	In other words, $p$ exists as the least element (with respect to $OP(V)$) in the set
	$$OP_v(V) := \{ r \in OP(V): v \in V_r \} = \{ r \in OP(V): v \perp_{\infty}^a (e - r) \} .$$
	
	We say that $OP(V)$ covers $V$, if every element $v \in V$ has a cover in $OP(V)$ and $p_1 
	\perp_{\infty}^a p_2$ whenever $p_i$ is the cover of $v_i \in V^+$ in $OP(V)$ for $i = 1, 2$ with 
	$v_1 \perp_{\infty}^a v_2$. 
\end{definition}

The covering property also determines a lattice structure in $OP(V)$. 
\begin{proposition}
	Let $(V, e)$ be an absolute order unit space in which $OP(V)$ covers $V$. Then $OP(V)$ is a lattice in 
	the order structure of $V$ restricted to $OP(V)$.
\end{proposition}
\begin{proof} 
	Let $p_1, p_2 \in OP(V)$ and let $p \in OP(V)$ be the cover of $p_1 + p_2$. Then $p_1 + p_2 \le \Vert p_1 + p_2 \Vert p$ 
	so that $p_1 \le p_1 + p_2 \le \Vert p_1 + p_2 \Vert p$. Now, by the order unit property of $p$, we get 
	$p_1 \le \Vert p_1 \Vert p \le p$. Similarly we can show that $p_2 \le p$. Next, let $p_i \le q, i = 1, 2$ 
	for some $q \in OP(V)$. Then $p_i \perp_{\infty}^a q^{\prime}, i = 1, 2$ so that $p_1 + p_2 \perp_{\infty}^a q^{\prime}$ 
	by the additivity of $\perp_{\infty}^a$. Now by Lemma \ref{16}, $p_1 + p_2 \le \Vert p_1 + p_2 \Vert q$. 
	Since $p$ covers $p_1 + p_2$, we get $p \le q$. Thus $p = \sup \{ p_1, p_2 \}$. Next, as $p_1^{\prime}, 
	p_2^{\prime} \in OP(V)$, we have $r = \sup \{ p_1^{\prime}, p_2^{\prime} \} \in OP(V)$. Now, by a well 
	known trick, we can show that  $\inf \{ p_1, p_2 \} = r^{\prime}$. Hence $OP(V)$ is a lattice.
\end{proof}

\subsection{Construction of a spectral family} \hfill
\begin{definition}
	Let $(V, e)$ be an absolute order unit normed space and let $p \in OP(V)$. Then every $v \in AC(p)$ has a unique decomposition
	$$v = v_p + v_{p^{\prime}}$$
	where $v_p \in V_p$ and $v_{p^{\prime}} \in V_{p^{\prime}}$. This decomposition will be referred as 
	the $p$-decomposition of $v$ and we shall write $v_p = C_p(v)$ and $v_{p^{\prime}} = 
	C_p^{\prime}(v)$ for all $v \in AC(p)$. In particular, for $v \in AC(p)^+$, we have $C_p (v) = w \dot\wedge (\Vert w \Vert p)$ 
	for all $p \in OP(V)$. By Proposition \ref{21}, $C_p: AC(p) \to V_p$ is a 
	$\vert \cdot \vert$-preserving surjective linear projection of norm one. Further, $C_p + C_p^{\prime}$ is 
	the identity operator on $AC(u)$.
\end{definition}
\begin{proposition}\label{43}
	Let $(V, e)$ be an absolute order unit space. If 
	$p, q \in OP(V)$ with $p \le q$, then $C_p, C_q, C_p^{\prime}$ and $C_q^{\prime}$ commute mutually 
	when restricted to $AC(p, q)$. In this case, $C_p C_q = C_p$ and consequently, $C_p C_q^{\prime} = 
	0$, $C_p^{\prime} C_q = C_q - C_p$ and $C_p^{\prime} C_q^{\prime} = C_q^{\prime}$ on $AC(p, q)$.
\end{proposition}
\begin{proof}
	Let $u \in AC(p, q)$. Then
	$$u = C_p(u) + C_p^{\prime}(u) = C_q(u) + C_q^{\prime}(u).$$
	As $p \le q$, we have $V_p \subset V_q \subset OC(q)$. Thus $C_p(u) \in AC(q)$ and consequently, 
	$C_p^{\prime}(u) = u - C_q(u) \in OC(q)$. Now as $AC(p^{\prime}) = AC(p), AC(q^{\prime}) = AC(q)$ 
	and $q^{\prime} \le p^{\prime}$ (for $p \le q$), we may conclude, by the dual arguments, that 
	$C_p(u), C_q(u), C_p^{\prime}(u), C_q^{\prime}(u) \in AC(p, q)$ whenever $u \in AC(p, q)$.
	
	Next, note that for $u \in AC(p, q)$ we have $C_p(u) \in V_p \subset V_q$ so that $C_q(C_p(u)) = 
	C_p(u)$. Dually, $C_p^{\prime}(C_q^{\prime}(u)) = C_q^{\prime}(u)$ so that $C_q C_p = C_p$ and 
	$C_p^{\prime} C_q^{\prime} = C_q^{\prime}$ on $AC(p, q)$. Now, if we recall that $C_p^{\prime} = I - C_p$ 
	and that $C_q^{\prime} = I - C_q$ on $AC(p, q)$ where $I$ is the identity operator on $AC(p, q)$, the 
	remaining facts can be verified in a routine way.
\end{proof}
Throughout this subsection, we shall assume that $(V, e)$ is an absolute order unit space  in which $OP(V)$ 
covers $V$ unless stated otherwise. We fix the following notations. Let $v \in V$ and $\alpha \in \mathbb{R}$. We write 
$c_{p}^{\pm}(v, \alpha )$ for the cover of $(v - \alpha p)^{\pm}$ in $OP(V)$ respectively, for any $p \in OP(V)$. For 
$p = e$, we shall simply write $c^{\pm}(v, \alpha)$. Thus $c^+(v, \alpha ) \perp c^-(v, \alpha)$. When 
$\alpha = 0$, we shall simply write $c^{\pm}(v)$ for $c^{\pm}(v, 0)$. Also as $(-v, -\alpha )^+ = (v, \alpha )^-$, 
we have, $c^+(-v, -\alpha ) = c^-(v, \alpha )$. 

Let $c(v)$ be the cover of $v$ in $OP(V)$. Then 
$$(v - \alpha e)^+ = (v - \alpha c(v))^+ + ( - \alpha c(v)^{\prime} )^+ \in 
V_{c(v)}^+ + V_{c(v)^{\prime}}^+$$
and
$$(v - \alpha e)^- = (v - \alpha c(v))^- + ( - \alpha c(v)^{\prime} )^- \in 
V_{c(v)}^+ + V_{c(v)^{\prime}}^+.$$
Thus
\begin{eqnarray*}
	(v - \alpha e)^+ &=& (v - \alpha c(v))^+, \quad \textrm{if} \quad \alpha \ge 0\\
	&=& (v - \alpha c(v))^+ - \alpha c(v)^{\prime}, \quad \textrm{if} \quad \alpha < 0
\end{eqnarray*}
and
\begin{eqnarray*}
	(v - \alpha e)^- &=& (v - \alpha c(v))^- + \alpha c(v)^{\prime}, \quad \textrm{if} \quad \alpha > 0\\
	&=& (v - \alpha c(v))^-, \quad \textrm{if} \quad \alpha \le 0.
\end{eqnarray*}
Now, it follows that
\begin{eqnarray*}
	(\dagger) \qquad
	c^+(v, \alpha ) &=& c_{c(v)}^+(v, \alpha) \quad \textrm{if} \quad \alpha \ge 0\\
	&=& c_{c(v)}^+(v, \alpha) + c(v)^{\prime} \quad \textrm{if} \quad \alpha < 0
\end{eqnarray*}
and
\begin{eqnarray*}
	(\ddagger) \qquad
	c^-(v, \alpha ) &=& c_{c(v)}^-(v, \alpha) + c(v)^{\prime} \quad \textrm{if} \quad \alpha > 0\\
	&=& c_{c(v)}^-(v, \alpha) \quad \textrm{if} \quad \alpha \le 0.
\end{eqnarray*}
In a similar manner, we can further conclude that
\begin{eqnarray*}
	(\dagger\dagger) \qquad
	c^+(v, \alpha ) &=& c_{c^+(v)}^+(v^+, \alpha) \quad \textrm{if} \quad \alpha \ge 0\\
	&=& c_{c^-(v)}^-(v^-, - \alpha) + c^-(v)^{\prime} \quad \textrm{if} \quad \alpha < 0
\end{eqnarray*}
and
\begin{eqnarray*}
	(\ddagger\ddagger) \qquad
	c^-(v, \alpha ) &=& c^-_{c^+(v)}(v^+, \alpha) + c^+(v)^{\prime} \quad \textrm{if} \quad \alpha > 0\\
	&=& c_{c^-(v)}^+(v^-, - \alpha) \quad \textrm{if} \quad \alpha \le 0.
\end{eqnarray*}
\begin{remark}\label{48}
	\begin{enumerate}
		\item For $v \in V^+$, $c^+(v, \alpha) = e$ whenever $ \alpha < 0$. Dually, $c^-(v, \alpha) = e$ 
		whenever $v \in - V^+$ and $\alpha > 0$.
		\item For $v \in V^+$, $c^-(v, \alpha) = 0$ whenever $ \alpha \le 0$. Dually, $c^+(v, \alpha) = 0$ 
		whenever $v \in - V^+$ and $\alpha \ge 0$.
	\end{enumerate}
\end{remark}
Now we prove some results which show that $\{ c^-(v, \alpha) : \alpha \in \mathbb{R} \}$ and 
$\{ c^+(v, \alpha)^{\prime} : \alpha \in \mathbb{R} \}$ are ``spectral'' families of projections for $v$.
\begin{proposition}\label{49}
	For $v \in V$, 
	the family $\{ c^-(v, \alpha), c^+(v, \alpha) : \alpha \in \mathbb{R} \}$ has the following properties:
	\begin{enumerate}
		\item For $\alpha < \beta$, $c^-(v, \alpha) \le c^-(v, \beta )$ and $c^+(v, \alpha) \ge c^+(v, \beta )$.
		\item  For $\alpha \le -\Vert v \Vert$, $c^-(v, \alpha) = 0$ and for $\alpha \ge \Vert v \Vert$,  $c^+(v, \alpha) = 0$.
		\item For $\alpha > \Vert v \Vert$, $c^-(v, \alpha) = e$ and for $\alpha < -\Vert v \Vert$,  $c^+(v, \alpha)
		= e$.
		\item $v \in AC(c^-(v, \alpha), c^+(v, \alpha); \alpha \in \mathbb{R})$.
		\item $C_{\alpha}^-(v) \le \alpha c^-(v, \alpha)$ and $C_{\alpha}^{-\prime}(v) \ge \alpha c^-(v, 
		\alpha)^{\prime}$; $C_{\alpha}^+(v) \ge \alpha c^+(v, \alpha)$ and $C_{\alpha}^{+\prime}(v) \le \alpha 
		c^+(v, \alpha)^{\prime}$ for each $\alpha \in \mathbb{R}$. Here $C_{\alpha}^{\pm} := C_{c^{\pm}(v, 
			\alpha)}$.
	\end{enumerate}
\end{proposition}
\begin{proof}
	Since $c^+(v, \alpha) = c^-(-v, -\alpha)$, we need to prove results only related to $c^-(v, \alpha)$.
	
	(1). Let $\alpha < \beta$ and set $\lambda = \beta - \alpha > 0$. Put $v_{\beta} = v - \beta e$ so 
	that 
	$c^-(v, \beta)$ is the cover of $v_{\beta}^-$. As $v_{\beta}^- \perp_{\infty}^a v_{\beta}^+$, we 
	see that $v_{\beta}^+ \perp_{\infty}^a c^-(v, \beta)$. Thus
	\begin{eqnarray*}
		(v - \alpha e)^- &=& (v_{\beta} + \lambda e)^-\\
		&=& (v_{\beta}^+ - v_{\beta}^- + \lambda c^-(v, \beta) + \lambda c^-(v, \beta)^{\prime} )^-\\
		&=& (v_{\beta}^+ + \lambda c^-(v, \beta)^{\prime} )^- + (- v_{\beta}^- + \lambda c^-(v, \beta))^-\\
		&=& (- v_{\beta}^- + \lambda c^-(v, \beta))^- \in V_{c^-(v, \beta)}^+.
	\end{eqnarray*}
	Therefore, $c^-(v, \alpha) \le c^-(v, \beta)$, if $\alpha < \beta$. 
	
	(2). Let $\alpha \le -\Vert v \Vert$. Then $v - \alpha e  \in V^+$ so that $(v - \alpha e)^- = 0$. Thus 
	$c^-(v, \alpha) = 0$.
	
	(3). Let $\alpha > \Vert v \Vert$. Set $k = (\alpha - \Vert v \Vert ) > 0$. Then 
	$$(v - \alpha e)^- = \alpha e - v  = (\Vert v \Vert e - v) + k e \ge k e.$$
	As $k > 0$, we conclude that $e \in V_{c^-(v, \alpha)}^+$. Thus $e \le c^-(v, \alpha)$ and consequently, 
	$c^-(v, \alpha) = e$.
	
	(4). Fix $\alpha \in \mathbb{R}$. Then
	$$v - \alpha e = (v - \alpha c(v))^+ - (v - \alpha c(v))^- - \alpha c(v)^{\prime} \in V_{\bar{u}_{\alpha}^+} 
	+ V_{\bar{u}_{\alpha}^-}  + V_{c(v)^{\prime}}.$$
	If $\alpha > 0$, then $c^-(v, \alpha) = \bar{u}_{\alpha}^- + c(v)^{\prime}$ and $c^+(v, \alpha) = \bar{u}_{\alpha}^+$. As $\bar{u}_{\alpha}^-, \bar{u}_{\alpha}^+$ and $c(v)^{\prime}$ are mutually absolutely 
	$\infty$-orthogonal with $\bar{u}_{\alpha}^- + \bar{u}_{\alpha}^+ + c(v)^{\prime} \le e$, we get that 
	$\bar{u}_{\alpha}^+ \le c^-(v, \alpha)^{\prime}$. Thus $V_{\bar{u}_{\alpha}^+} + V_{\bar{u}_{\alpha}^-}  
	+ V_{c(v)^{\prime}} \subset V_{c^-(v, \alpha)} + V_{c^-(v, \alpha)^{\prime}}$ whence $v - \alpha e \in 
	AC(c^-(v, \alpha))$. As $e \in AC(c^-(v, \alpha))$ and as the later is a subspace of $V$, we further 
	conclude that $v \in AC(c^-(v, \alpha))$. The other cases may be proved in a similar way.
	
	(5). Let $\alpha \in \mathbb{R}$. Then $C_{\alpha}^-(v - \alpha e) = - (v - \alpha e)^-$ and 
	$C_{\alpha}^-(e) = c^-(v, \alpha)$. Thus $C_{\alpha}^-(v) = - (v - \alpha e)^- + \alpha c^-(v, \alpha) 
	\le \alpha c^-(v, \alpha)$. We can prove the other statements In a similar manner.
\end{proof}
\begin{theorem}\label{50}
	Let $v \in AC(p)$ for some $p \in OP(V)$. Then for any $\alpha \in \mathbb{R}$, both $c^+(v, 
	\alpha)$ and $c^-(v, \alpha)$ are absolutely compatible with $p$.
\end{theorem}
\begin{proof}
	First, let $v \in V^+$ and $\alpha \ge 0$. As $v \in AC^+(p)$, we have $v = C_p(v) + C_p^{\prime}(v)$ 
	with $C_p(v) \in V_p^+$ and $C_p^{\prime} \in V_p^{\prime}$. Thus 
	\begin{eqnarray*}
		(v - \alpha e)^+ &=& ( (C_p(v) - \alpha p) + (C_p^{\prime}(v) - \alpha p^{\prime}) )^+\\ 
		&=& (C_p(v) - \alpha p)^+ + (C_p^{\prime}(v) - \alpha p^{\prime})^+.
	\end{eqnarray*}
	Also
	\begin{eqnarray*}
		(C_p(v) - \alpha e)^+ &=& (C_p(v) - \alpha p - \alpha p^{\prime})^+\\ 
		&=& (C_p(v) - \alpha p)^+ + ( - \alpha p^{\prime})^+\\
		&=& (C_p(v) - \alpha p)^+
	\end{eqnarray*}
	and
	\begin{eqnarray*}
		(C_p^{\prime}(v) - \alpha e)^+ &=& (C_p(v) - \alpha p - \alpha p^{\prime})^+\\ 
		&=& (C_p^{\prime}(v) - \alpha p^{\prime})^+ + ( - \alpha p)^+\\
		&=& (C_p^{\prime}(v) - \alpha p^{\prime})^+.
	\end{eqnarray*}
	Thus $(C_p(v) - \alpha e)^+ \perp_{\infty}^a (C_p^{\prime}(v) - \alpha e)^+$ so that
	$c^+(C_p(v), \alpha) \perp_{\infty}^a c^+(C_p^{\prime}(v), \alpha)$ with $c^+(C_p(v), \alpha) + 
	c^+(C_p^{\prime}(v), \alpha) = c^+(v, \alpha)$. Since $c^+(C_p(v), \alpha) \in V_p^+$ and  
	$c^+(C_p^{\prime}(v), \alpha) \in V_{p^{\prime}}^+$, we further see that $c^+(C_p(v), \alpha) 
	\le c^+(v, \alpha) \wedge p$ and that  $c^+(C_p^{\prime}(v), \alpha) \le c^+(v, \alpha) \wedge p^{\prime}$. Thus
	\begin{eqnarray*}
		c^+(v, \alpha) &=& c^+(C_p(v), \alpha) + c^+(C_p^{\prime}(v), \alpha)\\
		&\le& c^+(v, \alpha) \wedge p + c^+(v, \alpha) \wedge p^{\prime}\\
		&\le& c^+(v, \alpha).
	\end{eqnarray*}
	Now, it follows that $c^+(v, \alpha) \in AC^+(p)$ and that $c^+(C_p(v), \alpha) = c^+(v, \alpha) \wedge p$ 
	and $c^+(C_p^{\prime}(v), \alpha) = c^+(v, \alpha) \wedge p^{\prime}$. Now, as $c^+(v, \alpha) = e$ 
	whenever $v \in V^+$ and $\alpha < 0$, we get that $c^+(v, \alpha) \in AC^+(p)$ for any 
	$v \in AC^+(p)$ and $\alpha \in \mathbb{R}$. Now, since $v - \alpha e = (v + \Vert v \Vert e) 
	- (\alpha + \Vert v \Vert) e$ and $(v + \Vert v \Vert e) \in AC^+(p)$ whenever $v \in AC(p)$, we further 
	conclude that $c^+(v, \alpha) \in AC^+(p)$ whenever $v \in AC(p)$ and $\alpha \in \mathbb{R}$. Finally, as 
	$c^-(v, \alpha) = c^+( - v, - \alpha)$, the result holds.
\end{proof}
\begin{proposition}\label{51}
	Assume that 
	$v \in AC(p)$ for some $p \in OP(V)$ and that $\alpha \in \mathbb{R}$. Then $C_p(v) \le \alpha p$ and 
	$C_p^{\prime}(v) \ge \alpha p^{\prime}$ if and only if $c^-(v, \alpha) \le p \le c^+(v, \alpha)^{\prime}$. 
	Also in this case, $C_p(v - \alpha e) = - (v - \alpha e)^-$ and $C_p^{\prime}(v - \alpha e) = (v - \alpha 
	e)^+$.  
\end{proposition}
\begin{proof}
	Find $v_1$ and $v_2$ in $V^+$ such that $C_p(v) + v_1 = \alpha p$ and $C_p^{\prime}(v) - v_2 =  
	\alpha p^{\prime}$. Then $v_1 \in V_p^+$ and $v_2 \in V_{p^{\prime}}^+$ so that $v_1 
	\perp_{\infty}^a v_2$. Next, as $v \in AC(p)$, we have 
	$$v = C_p(v) + C_{p^{\prime}}^{\prime}(v) = (\alpha p - v_1) + (\alpha p^{\prime} + v_2) = \alpha e + 
	( v_2 - v_1).$$
	Thus $v - \alpha e = v_2 - v_1$ and consequently, $v_2 = (v - \alpha e)^+$ and $v_1 = (v - \alpha e)^-$  
	for $v_2 \perp_{\infty}^a v_1$. Now it follows that 
	$$(v - \alpha e)^+ = C_p^{\prime}(v) - \alpha p^{\prime} =  C_p^{\prime}(v - \alpha e)$$
	and
	$$(v - \alpha e)^- = - (C_p(v) - \alpha p) = - C_p(v - \alpha e).$$
	Further, we see that $(v - \alpha e)^+ \in V_{p^{\prime}}^+$ and that $(v - \alpha e)^- \in V_p^+$. Thus 
	$c^+(v, \alpha) \le p^{\prime}$ and $c^-(v, \alpha)^- \le p$ so that $c^-(v, \alpha) \le p \le c^+(v, 
	\alpha)^{\prime}$.
	
	Conversely, assume that $c^-(v, \alpha) \le p \le c^+(v, \alpha)^{\prime}$. Then by Proposition \ref{43}, 
	$C_{\alpha}^{+{\prime}} C_p = C_p = C_p C_{\alpha}^{+{\prime}}$ and $C_{\alpha}^{-{\prime}} 
	C_p^{\prime} = C_p^{\prime} = C_p^{\prime} C_{\alpha}^{-{\prime}}$. Thus 
	$$C_p(v - \alpha e) = C_p C_{\alpha}^{+{\prime}}(v - \alpha e) \le 0$$
	so that $C_p(v) \le \alpha p$ and 
	$$C_p^{\prime}(v - \alpha e) = C_p^{\prime} C_{\alpha}^{+{\prime}}(v - \alpha e) \ge 0$$
	so that $C_p^{\prime}(v) \ge \alpha p^{\prime}$.
\end{proof}
\begin{proposition}\label{52}
	Let $(V, e)$ be a monotone complete absolute order unit space in which $OP(V)$ covers $V$ and let 
	$v \in V$. Set $e_{\alpha} = c^+(v, \alpha )^{\prime}$ for each $\alpha \in \mathbb{R}$. Then $\wedge_{\alpha > 
		\alpha_0} e_{\alpha} = e_{\alpha_0}$ for each $\alpha_0 \in \mathbb{R}$.
\end{proposition}
\begin{proof}
	First note that $\wedge_{\alpha > \alpha_0} e_{\alpha}$ exists in $V$ as the later is monotone complete. We write 
	$\wedge_{\alpha > \alpha_0} e_{\alpha} = v_0$ so that $v_0 \in [0, e]$. Let $e_0 \in OP(V)$ be the cover of $v_0$. 
	Then $v_0 \le e_0$. Now, by the definition of the cover, $e_0 \le e_{\alpha}$ for each $\alpha > \alpha_0$ as 
	$v_0 \le e_{\alpha}$ for such $\alpha$. Thus $e_0 \le v_0$ and we have $e_0 = \wedge_{\alpha > \alpha_0} e_{\alpha} \in OP(V)$. 
	Also, then $e_{\alpha_0} \le e_0 \le e_{\alpha}$ if $\alpha > \alpha_0$. Fix $\alpha > \alpha_0$. 
	Then $C_{e_0} C_{e_{\alpha}} = C_{e_{\alpha}} C_{e_0} = C_{e_0}$ so that 
	$$C_{e_0}(v) = C_{e_0} C_{e_{\alpha}}(v) \le C_{e_0}( \alpha e_{\alpha}) = \alpha e_0.$$
	Thus $C_{e_0}(v) \le \alpha_0 e_0$. Similarly, as $e_{\alpha_0} \le e_0$, we have $e_0^{\prime} \le e_{\alpha_0}^{\prime}$. 
	Thus as above, we may get that $C_{e_0}^{\prime}(v) \ge \alpha_0 e_0^{\prime}$. Now, applying Proposition \ref{51}, 
	we can conclude that $e_0 = c^+(v, \alpha_0)^{\prime} := e_{\alpha_0}$.
\end{proof}
\begin{theorem}[Spectral Resolution]\label{53}
	Let $(V, e)$ be a monotone complete absolute order unit space in which $OP(V)$ covers $V$ and let 
	$v \in V$. Then there exists a unique family $\{ e_{\alpha}: \alpha \in \mathbb{R} \} \subset OP(V)$ such that
	\begin{enumerate}
		\item $\{ e_{\alpha}: \alpha \in \mathbb{R} \}$ is increasing;
		\item $e_{\alpha} = 0$ if $\alpha < - \Vert v \Vert$ and $e_{\alpha} = e$ if $\alpha \ge \Vert v \Vert$;
		\item $v \in OC(e_{\alpha}; \alpha \in \mathbb{R})$;
		\item $C_{e_{\alpha}}(v) \le \alpha e_{\alpha}$ and $C_{e_{\alpha}}^{\prime}(v) \ge \alpha e_{\alpha}^{\prime}$ 
		for each $\alpha \in \mathbb{R}$;
		\item If $v \in AC(p)$ for some $p \in OP(V)$ and if $C_p(v) \le \alpha p$ and $C_p^{\prime}(v) \ge \alpha p^{\prime}$, 
		then $p \le e_{\alpha}$; 
		\item $\wedge_{\alpha > \alpha_0} e_{\alpha} = e_{\alpha_0}$ for each $\alpha_0 \in \mathbb{R}$.
	\end{enumerate}
\end{theorem}
\begin{proof}
	Set $e_{\alpha} = c^+(v, \alpha )^{\prime}$ for each $\alpha \in \mathbb{R}$. Then (1), (2) and (3) follow from 
	Proposition \ref{49}; (4) and (5) follow from Proposition \ref{51} and (6) follows from Proposition \ref{52}. 
	Conversely, assume that a family $\{ e_{\alpha}: \alpha \in \mathbb{R} \} \subset OP(V)$ satisfies conditions (1) -- (6). Then 
	by Proposition \ref{51}, condition (4) yields $c^-(v, \alpha ) \le e_{\alpha} \le c^+(v, \alpha )^{\prime}$ and condition (5) yields 
	$c^+(v, \alpha )^{\prime}  \le e_{\alpha}$ for each $\alpha \in \mathbb{R}$. This completes the proof.
\end{proof}
\begin{definition}
	Let $(V, e)$ be a monotone complete absolute order unit space in which $OP(V)$ covers $V$ and let $v \in V$.
	Then an increasing family $\{ e_{\alpha}: \alpha \in \mathbb{R} \} \subset OP(V)$ is called the spectral resolution of $v$ in 
	$OP(V)$ if the family satisfies conditions (1) -- (6) of Theorem \ref{53}.
\end{definition}
\begin{theorem}[Spectral Decomposition]\label{54}
	Let $(V, e)$ be a monotone complete absolute order unit space in which $OP(V)$ covers $V$ and let $v \in V$.
	Consider the spectral resolution $\{ e_{\alpha} : \alpha \in \mathbb{R} \}$ of $v$ in $OP(V)$. Then for any $\epsilon > 0$ 
	and a finite increasing sequence $\alpha_0  < \dots < \alpha_n$ with $\alpha_0 < - \Vert v \Vert$,$\alpha_n > \Vert v \Vert$ and 
	$\max \{ \alpha_i - \alpha_{i-1} : 1 \le i \le n \} < \epsilon$, we have 
	$$\big\Vert v - \sum_{i=1}^n \alpha_i (e_{\alpha_i} - e_{\alpha_{i-1}}) \big\Vert < \epsilon .$$
\end{theorem}
\begin{proof}
	Let $\alpha < \beta$. Then $e_{\alpha} \le e_{\beta}$ so that by Proposition \ref{48}(5), we get 
	$$C_{\beta}(v) \le \beta e_{\beta} = \beta C_{\beta}(e) ~ \textrm{and} ~ C_{\alpha}^{\prime}(v) \ge 
	\alpha e_{\alpha}^{\prime} = \alpha C_{\alpha}^{\prime}(e).$$
	As $C_{\alpha}^{\prime}$ and $C_{\beta}$ are positive operators on $OC(e_{\alpha}, e_{\beta})$, we 
	get 
	$$\alpha C_{\beta} C_{\alpha}^{\prime}(e) \le C_{\beta} C_{\alpha}^{\prime}(v) = C_{\alpha}^{\prime} 
	C_{\beta}(v) \le \beta C_{\alpha}^{\prime} C_{\beta}(e).$$
	Since $\alpha < \beta$, by Proposition \ref{43}, we see that
	$$C_{\beta} C_{\alpha}^{\prime}(e) = C_{\alpha}^{\prime} C_{\beta}(e) = C_{\beta}(e) - C_{\alpha}(e) 
	= e_{\beta} - e_{\alpha}$$
	and 
	$$C_{\beta} C_{\alpha}^{\prime}(v) = C_{\alpha}^{\prime} C_{\beta}(v) = C_{\beta}(v) - C_{\alpha}(v)$$
	Thus we obtain
	$$\alpha (e_{\beta} - e_{\alpha}) \le (C_{\beta}(v) - C_{\alpha}(v)) \le \beta (e_{\beta} - e_{\alpha}).$$
	Applying this for $\beta = \alpha_i$ and $\alpha = \alpha_{i-1}$, where $1 \le i \le n$ and adding 
	together, we get  
	$$\sum_{i=1}^n \alpha_{i-1} (e_{\alpha_i} - e_{\alpha_{i-1}}) \le C_{\alpha_n}(v) - C_{\alpha_0}(v) \le 
	\sum_{i=1}^n \alpha_i (e_{\alpha_i} - e_{\alpha_{i-1}}).$$
	Now, $e_{\alpha_n} = e$ and $e_{\alpha_0} = 0$ so that $C_{\alpha_n}(v) = v$ and $C_{\alpha_0}(v) 
	= 0$. Also, $(e_{\alpha_i} - e_{\alpha_{i-1}}) \perp_{\infty}^a (e_{\alpha_j} - e_{\alpha_{j-1}})$  if 
	$i \not= j$ so that 
	\begin{eqnarray*}
		\big\Vert \sum_{i=1}^n \alpha_i (e_{\alpha_i} - e_{\alpha_{i-1}}) &-& \sum_{i=1}^n \alpha_{i-1} 
		(e_{\alpha_i} - e_{\alpha_{i-1}}) \big\Vert \\
		&=& \big\Vert \sum_{i=1}^n (\alpha_i - \alpha_{i-1}) 
		(e_{\alpha_i} - e_{\alpha_{i-1}}) \big\Vert \\
		&=& \max \{ \alpha_i - \alpha_{i-1} : 1 \le i \le n \} < \epsilon .
	\end{eqnarray*} 
	Thus $$\big\Vert v - \sum_{i=1}^n \alpha_i (e_{\alpha_i} - e_{\alpha_{i-1}}) \big\Vert < \epsilon .$$
\end{proof}

\bibliographystyle{amsplain}

\end{document}